\documentclass[12pt,a4paper]{article}

\usepackage[english]{babel} 
\usepackage[latin1]{inputenc} 
\usepackage[T1]{fontenc} 
\usepackage{lmodern} 
\usepackage{amsmath} 
\usepackage[fixamsmath]{mathtools} 
\usepackage{amssymb} 
\usepackage{amsfonts} 
\usepackage{amsthm} 
\usepackage[svgnames]{xcolor} 
\usepackage{geometry} 
\usepackage{tikz} 
\usepackage{pgfplots} 
\usepackage{epstopdf} 
\usepackage[numbers]{natbib} 

\usepackage{hyperref} 

\hypersetup{
    colorlinks=true,
    linkcolor={MediumBlue}, 
    citecolor={MediumBlue},
    filecolor=magenta,
    urlcolor=cyan
}

\pgfplotsset{compat=1.11}

\newcommand{\N}{\mathbb{N}} 
\newcommand{\R}{\mathbb{R}} 

\newcommand{\cH}{\mathcal{H}} 
\newcommand{\cP}{\mathcal{P}} 
\newcommand{\cN}{\mathcal{N}} 
\newcommand{\cM}{\mathcal{M}} 
\newcommand{\scalprod}[2]{ \left< #1 , #2 \right> }
\newcommand{\norm}[1]{ \left\| #1 \right\| }
\newcommand{\Hscalprod}[2]{ \scalprod{#1}{#2}_{\cH} }
\newcommand{\Hnorm}[1]{ \norm{#1}_{\cH} }

\newcommand{\colemph}[1]{{\textcolor{DarkGreen}{#1}}} 
\newcommand{\set}[1]{ \left\{ #1 \right\} }
\newcommand{\bK}{K} 
\newcommand{\balpha}{\boldsymbol{\alpha}} 
\newcommand{\kron}{\otimes}

\DeclareMathOperator{\id}{id}
\DeclareMathOperator{\myspan}{span}

\DeclareMathOperator{\rank}{rank}
\DeclareMathOperator{\diag}{diag}
\DeclareMathOperator{\myvec}{vec}
\DeclareMathOperator{\sign}{sign}
\DeclareMathOperator{\mynull}{null}

\newtheorem{theorem}{Theorem}[section]
\newtheorem*{theorem*}{Theorem}
\newtheorem{lemma}[theorem]{Lemma}
\newtheorem{corollary}[theorem]{Corollary}

\theoremstyle{definition}
\newtheorem{definition}[theorem]{Definition}
\newtheorem{example}[theorem]{Example}
\theoremstyle{remark}
\newtheorem{remark}[theorem]{Remark}

\usepackage{todonotes}

\usepackage[toc,page]{appendix}

\title{Interpolation with uncoupled separable matrix-valued kernels}
\author{
  D. Wittwar\thanks{Institute for Applied Analysis and Numerical Simulation, University of Stuttgart, Germany}
  \and
  G.~Santin\footnotemark[1]
  \and
  B.~Haasdonk\footnotemark[1]
}

\begin{document}
\maketitle


\begin{abstract}
In this paper we consider the problem of approximating vector-valued functions over a domain $\Omega$.
For this purpose, we use matrix-valued reproducing kernels, which can be related to Reproducing kernel
Hilbert spaces of vectorial functions and which can be viewed as an extension to the scalar-valued case.
These spaces seem promising, when modelling correlations between the target function components, as the components
are not learned independently of one another. We focus on the interpolation with such matrix-valued kernels.
We derive error bounds for the interpolation error in terms of a generalized power-function and
we introduce a subclass of matrix-valued kernels whose power-functions can be traced back to the power-function 
of scalar-valued reproducing kernels. Finally, we apply these kind of kernels to some artificial data to illustrate 
the benefit of interpolation with matrix-valued kernels in comparison to a componentwise approach.
\end{abstract}



\section{Introduction} \label{sec: Introduction}

Kernel methods are useful tools for dealing with a wide variety of different tasks ranging from 
machine learning e.g. via Support Vector Machines (SVMs) (\cite{Bishop2006,Schaback2006,Steinwart2011}),
function approximation from scattered data (\cite{Franke1979,Micchelli1986}) and many more. Especially the approximation 
aspect can be employed for generating surrogate models to speed up expensive function evaluation, see \cite{Wirtz2015a}.
In cases where the given output data or the desired target function is vector-valued, simple approaches which build individual
models for each function component can still be very costly, if the output is high dimensional and the component models rely on
independent data sets such that the union of those results in overly large sets. Additionally, approximating a vectorial
function componentwise with identical ansatz spaces might be the wrong choice, e.g. in case of different frequencies. 
We thus propose the use of matrix-valued kernels which lead to surrogates that can deal with correlations between function 
components, respective structural properties of the target function, and therefore provide a more suitable model. For divergence-free kernels,
matrix-valued kernel approximations have already been succesfully applied, see e.g.\ (\cite{Wendland2017b,Lowitzsch2005a,MR1254147,Fuselier2009}).

This paper is structured as follows: In Section \ref{section: RKHS} we begin with an introduction to matrix-valued kernels and extend well-known properties from the scalar-valued case
including error estimation. We then introduce a new subclass of matrix-valued kernels and study its properties in relation to the power-function
which enables us to perform a-priori interpolation error estimation in Section \ref{sec: Separable matrix-valued kernels}. 
A numerical example in Section \ref{sec: Numerical Example} illustrates the benefits of the matrix-valued ansatz when compared
to the scalar-valued case. Finally, we conclude with some remarks and an outlook.



\section{Reproducing kernel Hilbert spaces for matrix-valued kernels} \label{section: RKHS}

In this section we want to give a short overview over the theory of matrix-valued kernels and their application in interpolation. As matrix-valued kernels are an
extension of the well studied scalar-valued kernels, many of the following notions, properties and concepts are again suitable extensions of their scalar-valued counterparts.
For a more extensive overview with regards to this topic and other approximation schemes involving matrix-vaued kernels such as regression, we refer to literature, e.g.\ 
\cite{Alvarez2012,Micchelli2005,Reisert2007}.
\begin{definition}[Matrix-valued kernel] \label{def: Matrix-valued kernel}
 Let $\Omega$ be a non empty set. We call a function $k : \Omega \times \Omega \rightarrow \R^{m\times m}$ a \colemph{matrix-valued kernel} if
 \begin{equation*}
  k(x,y) = k(y,x)^T \quad \forall\, x,y \in \Omega.
 \end{equation*}
\end{definition}

\begin{definition}[Reproducing kernel Hilbert space (RKHS)] \label{def: RKHS}
Let $\cH$ denote a Hilbert space of $\R^m$-valued functions over a domain $\Omega$  with inner product
$\Hscalprod{\cdot}{\cdot}$ and induced norm $\Hnorm{ \cdot }$. We call $\cH$ an \colemph{ $\R^m$-reproducing kernel Hilbert space ($\R^m$-RKHS)}, if for all 
$x \in \Omega$ and $\alpha \in \R^m$ the directional point evaluation functional $\delta_x^{\alpha} : \cH \rightarrow \R$
defined by
\begin{equation}
\delta_x^{\alpha}(f) := f(x)^T\alpha. \label{def:eq: Riesz Representation}
\end{equation}
is bounded, i.e.\
\begin{equation*}
 \|\delta_x^{\alpha}\|_{\cH^{\prime}} := \sup\limits_{f \in \cH \setminus \set{0} } \frac{\delta_x^{\alpha}(f)}{ \Hnorm{f} } < \infty.
\end{equation*}

\end{definition}

Similar to the scalar-valued case, see for example \cite{Aronszajn1950}, there exists a one-to-one correspondence between RKHS of vector-valued functions and positive definite matrix-valued kernels. A necessary
concept for this is the notion of positive definiteness which is a straightforward extension from the scalar-valued case and is given as follows:

\begin{definition}[Definiteness]\label{def:Definiteness}
 Let $\Omega$ be non empty and $k : \Omega \times \Omega \rightarrow \R^{m\times m}$ be a matrix-valued kernel. 
 For a finite set $X := \set{x_1,\dots,x_n} \subset \Omega$, $n \in \N$, we define the \colemph{Gramian matrix}
 $\bK \in \R^{mn \times mn}$ as the block matrix given by
 \begin{equation}
  \bK := k(X,X) := (k(x_i,x_j))_{i,j=1}^{n} = \begin{bmatrix}
          k(x_1,x_1) &\cdots& k(x_1,x_n) \\
          \vdots & \ddots & \vdots \\
          k(x_n,x_1) & \cdots & k(x_n,x_n)
         \end{bmatrix}.\label{def:eq: Gramian matrix}
 \end{equation} 
The kernel $k$ is denoted as \colemph{positive definite}, if for all $n \in \N$ and $X = \set{x_1,\dots,x_n} \subset \Omega$ the Gramian matrix $\bK$ is positive
semi-definite, i.e.\ it holds
\begin{equation}
 \alpha^T\bK\alpha \geq 0 \quad \forall \, \alpha \in \R^{mn}. \label{def:eq:PD}
\end{equation}
The kernel is called \colemph{strictly positive definite (s.p.d.)} if for all $n \in \N$ and pairwise distinct $X = \set{x_1,\dots,x_n} \subset \Omega$ the 
Gramian matrix $\bK$ is positive definite, i.e.\ it holds
\begin{equation}
 \alpha^T \bK \alpha > 0 \quad \forall \, \alpha \in \R^{nm} \setminus\{0\}. \label{def:eq:SPD}
\end{equation}
Furthermore, we will introduce the abbreviation
\begin{equation}
 k(X,x)^T := k(x,X) := (k(x,x_i))_{i=1}^n 
 = \begin{bmatrix}
    k(x,x_1) & \cdots & k(x,x_n) \label{def:eq: k(x,X)}
   \end{bmatrix}
 \in \R^{m \times mn}
\end{equation}
as it will be useful later on.
\end{definition}
Going forward, for $A,B $ symmetric matrices, we will use the notation $A \succeq B$ if $A-B$ is positive semi-definite and $A \succ B$ if $A-B$ is positive definite.

As mentioned before, every RKHS corresponds to a positive definite matrix-valued kernel and vice versa. We state this in the following theorem. 
A proof for operator-valued kernels, which include the finite dimensional matrix-case, can be found for example in \cite{KDPCRA2016}.

\begin{theorem}[One-to-one correspondence]
Let $\cH$ be an $\R^m$-RKHS. Then there exists a unique positive definite matrix-valued kernel
$k : \Omega \times \Omega \rightarrow \R^{m \times m}$ such that for all $x \in \Omega$, $\alpha \in \R^m$ and $f \in \cH$
\begin{equation}
k(\cdot,x)\alpha \in \cH \quad \text{and} \quad \Hscalprod{f}{k(\cdot,x)\alpha} = f(x)^T\alpha. \label{thm:eq: Reproducing Kernel Condition}
\end{equation}
Conversely, if $k : \Omega \times \Omega \rightarrow \R^{m \times m}$ is a positive definite matrix-valued kernel, then there exists a unique
Hilbert space $\cH$ of $\R^m$-valued functions on $\Omega$ such that \eqref{thm:eq: Reproducing Kernel Condition} holds.
\end{theorem}

In the scalar-valued case, there is an alternative characterization by means of feature maps, i.e.\ for a p.d. kernel there exists a mapping 
$\Phi : \Omega \rightarrow V$, where $V$ is some Hilbert space, such that the reproducing kernel is given by
\begin{equation*}
 k(x,y) = \langle \Phi(x),\Phi(y) \rangle_V, \quad \forall \, x,y \in \Omega.
\end{equation*}
In the matrix-valued case this is no longer possible, as inner products are scalar-valued. Nonetheless, the concept can be adapted by allowing mappings 
$\Phi : \Omega \times \R^m \rightarrow V$ such that
\begin{equation*}
 \alpha^Tk(x,y)\beta = \langle \Phi(x,\alpha), \Phi(y,\beta) \rangle_V ,\quad \forall \, x,y \in \Omega, \, \alpha,\beta \in \R^m.
\end{equation*}
For further details we refer to \cite{CMPY2008,MP2004}.

\begin{lemma}[Closed subspaces are RKHS] \label{lem: Closed subspaces are RKHS}
 Let $\cH$ be an $\R^m$-RKHS. If $\cN \subset \cH$ is a closed subspace then $\cN$ is also an $\R^m$-RKHS. Furthermore, if $\cN$ is finite dimensional with
 orthonormal basis $(v_n)_{n = 1}^{\dim \cN}$, then the reproducing kernel $k_{\cN}: \Omega \times \Omega \rightarrow \R^{m \times m}$ of $\cN$ is given by
 \begin{equation}
  k_{\cN}(x,y) = \sum\limits_{n=1}^{\dim\cN} v_n(x)v_n(y)^T. \label{lem:eq: reproducing kernel subspace}
 \end{equation}
\end{lemma}

\begin{proof}
By Definition \ref{def: RKHS} it is sufficient to show that the directional point evaluation functionals $\delta_x^{\alpha} : \cN \rightarrow \R$ are bounded:
\begin{equation*}
 \| \delta_x^{\alpha} \|_{\cN^{\prime} } = \sup\limits_{f \in \cN \setminus \set{0} } \frac{\delta_x^{\alpha}(f)}{\|f\|_{\cN}} 
 \leq \sup\limits_{f \in \cH \setminus \set{0} } \frac{\delta_x^{\alpha}(f)}{ \Hnorm{f} } =  \| \delta_x^{\alpha} \|_{\cH^{\prime} } < \infty.
\end{equation*}
 We now show that $k_{\cN}$ as defined in \eqref{lem:eq: reproducing kernel subspace} satisfies the reproducing property \eqref{thm:eq: Reproducing Kernel Condition}:
 Let $\alpha \in \R^m$, it holds
 \begin{equation*}
  k_{\cN}(\cdot,x)\alpha = \sum\limits_{n=1}^{\dim\cN} v_n(\cdot) \underbrace{v_n(x)^T\alpha}_{ \in \R} \in \cN
 \end{equation*}
  and for $v_i$ with $i = 1,\dots,\dim\cN$
  \begin{equation*}
   \Hscalprod{ v_i}{ k_{\cN}(\cdot,x)\alpha} = \sum\limits_{n=1}^{\dim\cN} \underbrace{\Hscalprod{v_i}{v_n}}_{= \delta_{in}}v_n(x)^T\alpha = v_i(x)^T\alpha.
  \end{equation*}
 Due to the linearity of the inner product $\Hscalprod{f}{k_{\cN}(\cdot,x)\alpha} = f(x)^T\alpha$ holds for all $f \in \cN$.
\end{proof}

In general we are interested in finite dimensional subspaces of $\cH$ which are spanned by kernel evaluations $k(\cdot,x_i)\alpha$ for different centers $x_i \in X = \set{ x_1,\dots,x_N} \subset \Omega$
and directions $\alpha \in \R^m$, i.e.\ we
are considering subspaces $\cN(X) \subset \cH$ of the form
\begin{equation}
 \cN(X) := \myspan \set{ k(\cdot,x_i)\alpha | \, x_i \in X, \, \alpha \in \R^m }. \label{eq: Subspace of center translations}
\end{equation}
Caused by the reproducing property \eqref{thm:eq: Reproducing Kernel Condition} the orthogonal projection operator $\Pi_{\cN(X)} : \cH \rightarrow \cN(X)$, which is characterized by
\begin{equation}
 \Hscalprod{ f - \Pi_{\cN(X)}(f) }{ g} = 0, \quad \forall \, g \in \cN(X), \label{eq: orthogonal projection}
\end{equation}
coincides with the interpolation operator $I_{\cN(X)} : \cH \rightarrow \cN(X)$ which interpolates a given function $f \in \cH$ on the set $X$ by a function $g = I_{\cN(X)}(f) \in \cN(X)$, i.e.\
\begin{equation*}
 f(x_i) = I_{\cN(X)}(f)(x_i) = g(x_i), \quad \forall \, x_i \in X.
\end{equation*}
Indeed, using $g = k(\cdot,x_i)\alpha \in \cN(X)$ in \eqref{eq: orthogonal projection} results in
\begin{equation*}
 \left( f(x_i) - \Pi_{\cN(X)}(f)(x_i) \right)^T \alpha = 0, \quad \forall \, x_i \in X, \, \alpha \in \R^m,
\end{equation*}
In summary, this means that the interpolant
\begin{equation}
 \Pi_{\cN(X)}(f) = \sum\limits_{i=1}^N k(\cdot,x_i)\alpha_i \label{eq: Interpolant}
\end{equation}
is characterized by solutions of the linear system
\begin{equation}
 k(X,X)\balpha := \begin{bmatrix}
          k(x_1,x_1) &\cdots& k(x_1,x_N) \\
          \vdots & \ddots & \vdots \\
          k(x_n,x_1) & \cdots & k(x_N,x_N)
         \end{bmatrix}
         \begin{bmatrix}
          \alpha_1 \\
          \vdots \\
          \alpha_N
         \end{bmatrix}
 = \begin{bmatrix}
    f(x_1) \\
    \vdots \\
    f(x_N)
   \end{bmatrix} =: f(X). \label{eq: Interpolation System}
\end{equation}
If the kernel $k$ is strictly positive definite, system \eqref{eq: Interpolation System} admits a unique solution as the system matrix $k(X,X)$ is regular. Therefore, an interpolant is always
well defined even if the right hand side in \eqref{eq: Interpolation System} does not stem from the evaluation of a function $f \in \cH$ on the set of centers $X$. In cases where the kernel is only
positive definite, i.e.\ $k(X,X)$ is positive semi-definite, the system has in general no unique solution for arbitrary right hand sides. However, a solution still exists when $f \in \cH$:

\begin{lemma} \label{lem: Projection onto subspace}
 Let $k: \Omega \times \Omega \rightarrow \R^{m \times m}$ be a matrix-valued positive definite kernel, $X = \set{ x_1,\dots,x_n} \subset \Omega$, $f \in \cH$. Furthermore,
 let \begin{equation*}
      \Pi_{\cN(X)}(f) = \sum\limits_{i=1}^n k(\cdot,x_i) \alpha_i
     \end{equation*}
 be the orthogonal projection of $f$ onto $\cN(X)$, where
 $\balpha : = \begin{bmatrix}
\alpha_1^T & \cdots & \alpha_n^T
\end{bmatrix}^T \in \R^{mn}$.
Then it holds
\begin{equation}
 \balpha \in k(X,X)^{+}f(X) + \mynull( k(X,X)).
\end{equation}
Here $k(X,X)^{+}$ denotes the Moore-Penrose pseudo inverse of $k(X,X)$.
\end{lemma}

\begin{proof}
 Let ${e_1,\dots,e_m}$ denote the standard basis of $\R^m$. By \eqref{eq: orthogonal projection} the interpolant satisfies
 \begin{align*}
  f(x_l)^Te_j = \Hscalprod{ \Pi_{\cN(X)}(f) } {k(\cdot,x_l)e_j} & = \Hscalprod {k(\cdot,x_l)e_j}{ \Pi_{\cN(X)}(f) } \\
  & = \Hscalprod {k(\cdot,x_l)e_j}{ k(\cdot,X)\balpha } \\
  & = e_j^Tk(X,x_l)^T\balpha \\
  & = e_j^T k(x_l,X) \balpha.
  \end{align*}
 
 Since this holds for all $l = 1,\dots,n$ and $j = 1,\dots,m$ we conclude
 \begin{equation}
  k(X,X)\balpha = f(X), \label{lem:eq: Interpolation System}
 \end{equation}
 i.e.\ $\balpha$ solves \eqref{eq: Interpolation System}. Let $\balpha^{\ast} := k(X,X)^{+}f(X)$ and by use of \eqref{lem:eq: Interpolation System} we get
 \begin{equation*}
  k(X,X)\balpha^{\ast} = k(X,X)k(X,X)^{+}f(X) = k(X,X)k(X,X)^{+}k(X,X)\balpha = k(X,X)\balpha = f(X)
 \end{equation*}
 and therefore $\balpha^{\ast}$ also solves \eqref{eq: Interpolation System} which implies $\balpha - \balpha^{\ast} \in \mynull(k(X,X))$.
\end{proof}
Following the above property it seems reasonable to define an approximation to a given function $f$ in the subspace $\cN(X)$ by
\begin{equation*}
 g(x) := k(x,X)k(X,X)^{+}f(X) \approx f(x)
\end{equation*}
even if $f \notin \cH$. In this case the interpolation property at the centers $X$ can no longer be guaranteed 
as in the strictly positive definite case, as $f(X) \in \mathrm{range}(k(X,X))$ cannot be guaranteed.

Before we further investigate how the error between a function $f \in \cH$ and its interpolant $\Pi_{\cN(X)}(f)$ can be quantified, we will present a direct corollary in which we derive
an alternative representation of the reproducing kernel on $\cN(X)$:

\begin{corollary}[Reproducing kernel of $\cN(X)$] \label{cor: Reproducing kernel of NX}
It holds
\begin{equation*}
 k_{\cN(X)}(x,y) = k(x,X)k(X,X)^{+}k(X,y). 
\end{equation*}

\end{corollary}
\begin{proof}
 By Lemma  \ref{lem: Projection onto subspace} we have
 \begin{equation*}
  \Pi_{\cN(X)}k(\cdot,x)\alpha = k(\cdot,X)k(X,X)^{+}k(X,x)\alpha.
 \end{equation*}
  It is therefore sufficient to show that for any $\alpha \in \R^m$
 \begin{equation*}
  \Pi_{\cN(X)}k(\cdot,x)\alpha = k_{\cN(X)}(\cdot,x)\alpha.
 \end{equation*}
 To this end, we first show that $\Pi_{\cN} : \cH \rightarrow \cN(X)$ is self-adjoint. For this purpose, let $f,g \in \cH$, then it holds
 \begin{align*}
  \langle \Pi_{\cN(X)}f , g \rangle & = \underbrace{\langle  \Pi_{\cN(X)}f , g - \Pi_{\cN(X)}g \rangle}_{= 0} + \langle \Pi_{\cN(X)}f  , \Pi_{\cN(X)}g \rangle \\
  & = - \underbrace{\langle f - \Pi_{\cN(X)}f, \Pi_{\cN(X)}g \rangle}_{= 0} + \langle f, \Pi_{\cN(X)} g \rangle \\
  & = \langle f , \Pi_{\cN(X)}g \rangle.
 \end{align*}
 By definition of the projection operator $\Pi_{\cN(X)}k(\cdot,x)\alpha \in \cN(X)$ and by the above it holds for any $f \in \cN(X)$:
 \begin{equation*}
  \Hscalprod{ f}{\Pi_{\cN(X)}k(\cdot,x)\alpha} = \Hscalprod{ \Pi_{\cN(X)}(f)}{k(\cdot,x)\alpha} = \Hscalprod{f}{k(\cdot,x)\alpha} = f(x)^T\alpha.
 \end{equation*}
 \end{proof}
The above corollary extends a well known result for scalar-valued kernels, see \cite{Mouattamid2009}, which states that the reproducing kernel on a closed subspace is equal
to the projection of the reproducing kernel on the entire space with regard to either argument.
However, in the matrix-valued case this does not carry over immediately, as the kernel has to be weighted with a direction, since the kernel itself is not an element of the RKHS.

As a tool to measure the error between $f$ and its interpolant we want to present the so called power-function, which for example was used in \cite{Sch1993a} for scalar-valued kernels:

\begin{definition}[Power-function]
Let $\cH$ be an $\R^m$-RKHS and $\cN \subset \cH$ be a closed subspace. Furthermore, 
let $\Pi_{\cN} : \cH \rightarrow \cN$ denote the orthogonal projection onto $\cN$. We define the \colemph{power-function} $\cP_{\cN} : \cH^{\prime} \rightarrow \R$ by
\begin{equation}
\cP_{\cN}(\lambda) := \sup\limits_{f \in \cH\setminus\set{0}} \frac{| \lambda(f) - \lambda(\Pi_{\cN}(f))|}{\Hnorm{f}} \quad \text{for } \lambda \in \cH^{\prime}. \label{def:power-function}
\end{equation}
In the case where $\lambda = \delta_x^{\alpha}$, we might also use the notation
\begin{equation*}
\cP_{\cN}^{\alpha}(x) := \cP_{\cN}(\delta_x^{\alpha}).
\end{equation*}
\end{definition}
In other words, the power-function maps a linear operator $\lambda$ to the norm of the composition of $\lambda$ with the orthogonal projection onto $\cN^{\perp}$:
\begin{align}
 \cP_{\cN}(\lambda) & = \sup\limits_{f \in \cH\setminus\set{0}} \frac{| \lambda(f) - \lambda(\Pi_{\cN}(f))|}{\Hnorm{f}} \nonumber \\ 
		    & = \sup\limits_{f \in \cH\setminus\set{0}} \frac{| \lambda\circ (\id - \Pi_{\cN})(f) |}{\Hnorm{f}} \nonumber \\
		    & = \sup\limits_{f \in \cH\setminus\set{0}} \frac{| \lambda \circ \Pi_{\cN^{\perp}}(f))|}{\Hnorm{f}} \nonumber \\
		    & = \| \lambda \circ \Pi_{\cN^{\perp}} \|_{\cH^{\prime} }
\end{align}
We want to remark that the above definition of the power-function is, in contrast to the power-function introduced in \cite{Schroedl09}, independent of the function $f$ and can be utilized to derive
a-priori error bounds which we show in Corollary \ref{cor: Bound on the interpolation error}.\\
It is easy to see that for a nested sequence of closed subspaces $\cN_{1} \subset \cN_{2} \subset \dots$ the power-function is non-increasing, i.e.\ $P_{\cN_1}(\lambda) \geq P_{\cN_2}(\lambda) \geq \dots$.
For general $\lambda \in \cH^{\prime}$ the evaluation of $P_{\cN}(\lambda)$ is nontrivial, however, using the Riesz representer $v_{\lambda} \in \cH$ of $\lambda$ we obtain
an alternative representation of $\cP_{\cN}(\lambda)$:

\begin{corollary}[Alternative representation of the power-function] \label{cor: Alternative representation of the power-function}
Let $\cH$ be an $\R^m$-RKHS and $\cN \subset \cH$ be a closed subspace. Furthermore,
let $\Pi_{\cN} : \cH \rightarrow \cN$ denote the orthogonal projection onto $\cN$ and $\cP_{\cN} : \cH^{\prime} \rightarrow \R$ the 
power-function. For any $\lambda \in \cH^{\prime}$ let $v_{\lambda} \in \cH$ denote its Riesz representer. Then it holds
\begin{equation*}
\cP_{\cN}(\lambda) = \Hnorm{v_{\lambda} - \Pi_{\cN}(v_{\lambda})} = \Hnorm{\Pi_{\cN^{\perp}}(v_{\lambda})}.
\end{equation*}
\end{corollary}

\begin{proof}
It follows from the definition of the power-function \eqref{def:power-function}
\begin{equation*}
\cP_{\cN}(\lambda) = \sup\limits_{f \in \cH\setminus\set{0}} \frac{\Hscalprod{v_{\lambda}}{f - \Pi_{\cN}(f)}}{\Hnorm{f}}.
\end{equation*}
Since both $\Pi_{\cN}$ and $\id - \Pi_{\cN}$ are orthogonal projections by assumption and therefore self-adjoint, the Cauchy-Schwarz inequality yields
\begin{equation*}
\cP_{\cN}(\lambda) = \sup\limits_{f \in \cH\setminus\set{0}} \frac{\Hscalprod{v_\lambda - \Pi_{\cN}(v_\lambda)}{f}}{\Hnorm{f}} \leq \Hnorm{v_\lambda - \Pi_{\cN}(v_\lambda)},
\end{equation*}
and equality is reached for $f = v_{\lambda} - \Pi_{\cN}(v_{\lambda})$.
\end{proof}
For the directional point evaluation functional $\delta_{x}^{\alpha}$ the Riesz representer is given by the reproducing kernel $k(\cdot,x)\alpha$. Therefore, we can easily compute the power-function
using the reproducing property of $k$ on $\cH$ and $k_{\cN} $ on $\cN$:

\begin{corollary} \label{cor: Power Function}
For any $x \in \Omega$ and $\alpha \in \R^{m}$ it holds
\begin{equation*}
\cP_{\cN}^{\alpha}(x)^2 = \alpha^T\left(k(x,x) - k_{\cN}(x,x)\right)\alpha = \alpha^T k_{\cN^{\perp}}(x,x)\alpha.
\end{equation*}
\end{corollary}
\begin{proof}
By Corollary \ref{cor: Alternative representation of the power-function} it holds
\begin{align*}
\cP_{\cN}^{\alpha}(x)^2 &= \Hnorm{k(\cdot,x)\alpha - \Pi_{\cN}k(\cdot,x)\alpha}^2 \\
& = \Hnorm{k(\cdot,x)\alpha - k_{\cN}(\cdot,x)\alpha}^2 \\
& = \Hscalprod{k(\cdot,x)\alpha - k_{\cN}(\cdot,x)\alpha}{k(\cdot,x)\alpha - k_{\cN}(\cdot,x)\alpha}\\
& = \Hscalprod{k(\cdot,x)\alpha}{k(\cdot,x)\alpha} - 2 \Hscalprod{k(\cdot,x)\alpha}{k_{\cN}(\cdot,x)\alpha} + \Hscalprod{k_{\cN}(\cdot,x)\alpha}{k_{\cN}(\cdot,x)\alpha} \\
& = \alpha^Tk(x,x)\alpha -2 \alpha^Tk_{\cN}(x,x)\alpha +\alpha^Tk_{\cN}(x,x)\alpha \\
& = \alpha^T\left(k(x,x) - k_{\cN}(x,x)\right)\alpha \\
& = \alpha^T k_{\cN^{\perp}}(x,x)\alpha.
\end{align*}
Here the identity $\alpha^T(k(x,x) - k_{\cN}(x,x))\alpha = \alpha^Tk_{\cN^{\perp}}\alpha$ follows from  Corollary \ref{cor: Reproducing kernel of NX} as $\Pi_{\cN^{\perp}} = \id - \Pi_{\cN}$.
\end{proof}

\begin{corollary}[Bound on the interpolation error] \label{cor: Bound on the interpolation error}
Let $\cH$ be an $\R^m$-RKHS with reproducing kernel $k$, let $\cN \subset \cH$ be a closed subspace with reproducing kernel $k_{\cN}$ and $\Pi_{\cN} : \cH \rightarrow \cN$ the orthogonal projection onto $\cN$.
Then it holds for any $f \in \cH$ and $\alpha \in \R^m$
\begin{equation}
\left|\left( f(x) - (\Pi_{\cN}f)(x)\right)^T\alpha\right| \leq \cP_{\cN}^{\alpha}(x) \Hnorm{f - \Pi_{\cN}f} \leq \cP_{\cN}^{\alpha}(x) \Hnorm{f}  , \quad \forall \, x \in \Omega \label{eq: Directional error bound}
\end{equation}
and
\begin{align*}
 \| f(x) - (\Pi_{\cN}f)(x)\|_2 & \leq \| k(x,x)-k_{\cN}(x,x) \|_2^{1/2} \Hnorm{f - \Pi_{\cN}f}, \\
 \| f(x) - (\Pi_{\cN}f)(x)\|_{\infty} & \leq \max\limits_{i=1,\dots,m} |k(x,x)_{ii}-k_{\cN}(x,x)_{ii}|^{1/2} \Hnorm{f - \Pi_{\cN}f}, \\
 \| f(x) - (\Pi_{\cN}f)(x)\|_1 & \leq \sqrt{m} \| k(x,x)-k_{\cN}(x,x) \|_2^{1/2} \Hnorm{f - \Pi_{\cN}f}.
\end{align*}
Here $\| \cdot \|_2$ denotes the spectral norm on $\R^{m \times m}$.
\end{corollary}

\begin{proof}
 It holds
 \begin{align} \label{eq: directional bound}
  \left|\left( f(x) - (\Pi_{\cN}f)(x)\right)^T\alpha\right| & = \left|\Hscalprod{f - \Pi_{\cN}f}{k(\cdot,x)\alpha}\right| \nonumber \\
  & = \left|\Hscalprod{ (\id-\Pi_{\cN})f}{k(\cdot,x)\alpha}\right| \nonumber \\
  & = \left|\Hscalprod{ (\id-\Pi_{\cN})^2f}{k(\cdot,x)\alpha}\right| \nonumber \\
  & = \left|\Hscalprod{(\id-\Pi_{\cN})f}{(\id-\Pi_{\cN})k(\cdot,x)\alpha}\right| \\
  & \leq \Hnorm{f-\Pi_{\cN}f}\Hnorm{k(\cdot,x)\alpha - k_{\cN}(\cdot,x)\alpha} \nonumber \\
  & = \Hnorm{f-\Pi_{\cN}f} \cP_{\cN}^{\alpha}(x). \nonumber
 \end{align}
Choosing $\alpha = f(x) - (\Pi_{\cN}f)(x)$ and applying Corollary \ref{cor: Power Function} we get
\begin{align*}
 \| f(x) - (\Pi_{\cN}f)(x)\|_2^2 &  \leq   \left( \alpha^T \left(k(x,x)-k_{\cN}(x,x)\right)\alpha\right)^{1/2} \Hnorm{f-\Pi_{\cN}f}\\
				 & \leq  \| \alpha \|_2 \| (k(x,x)-k_{\cN}(x,x) \|_2^{1/2}\Hnorm{f-\Pi_{\cN}f}
\end{align*}
and after dividing by $\| \alpha \| = \| f(x) - (\Pi_{\cN}f)(x)\|_2$
\begin{equation*}
  \| f(x) - (\Pi_{\cN}f)(x)\|_2 \leq \| k(x,x)-k_{\cN}(x,x) \|_2^{1/2} \Hnorm{f-\Pi_{\cN}f}
\end{equation*}
Choosing $\alpha = e_i$ results in
\begin{equation*}
 |f(x)_i - \left(\Pi_{\cN}(f)(x)\right)_i|\leq \|f\|_{\cH} \cP_{\cN}^{e_i}(x) \leq |k(x,x)_{ii}-k_{\cN}(x,x)_{ii}|^{1/2} \Hnorm{f-\Pi_{\cN}f}
\end{equation*}
Maximization over $i \in \set{ 1,\dots,m}$ gives the desired bound.\\
For the last inequality the choice $\alpha = \sum\limits_{i=1}^m \sign(\left(f(x) - \Pi_{\cN}f(x)\right)^Te_i)e_i$ gives the desired result, as it holds  $\|\alpha\| = \sqrt{m}$.
\end{proof}




\section{Separable matrix-valued kernels} \label{sec: Separable matrix-valued kernels}

In order to practically solve interpolation problems, we need to take a look at how matrix-valued kernels can be constructed. To this end, we consider matrix-valued kernels which stem from
scalar-valued kernels. In particular, we focus on the notion of separable kernels, see \cite{Alvarez2012}, and we introduce a new subtype for which error estimation
via the power-function can be traced back to the power-functions of the scalar-valued kernels that were used to generate the matrix-valued kernel. For further details and different construction methods we refer 
to previous work, e.g. \cite{Beatson2011b,MG2014,Carmeli2006}, in this field.


\begin{definition}[Separable Kernels] \label{def: Separable Kernels}
Let $k : \Omega \times \Omega \rightarrow \R^{m \times m}$ be a matrix-valued kernel, let $Q_1,\dots,Q_p \in \R^{m \times m}$ be a collection of symmetric matrices
and $k_1,\dots,k_p : \Omega \times \Omega \rightarrow  \R$ a collection of scalar-valued kernels, such that
\begin{equation}
k(x,y) = \sum\limits_{i=1}^{p} k_i(x,y)Q_i, \quad \text{ for all } x,y \in \Omega. \label{def: separable kernel}
\end{equation}
We call $(k_i,Q_i)_{i=1}^p$ a \colemph{decomposition} of $k$ and $p$ its \colemph{length}. If $p$ is minimal then the kernel $k$ is called separable of order $p$.
\end{definition}

To guarantee the (strict) positive definiteness of the kernel $k$ further assumption on the scalar-valued kernels $k_i$ and symmetric matrices $Q_i$ have to be made.
Taking a closer look at the Gramian matrix $k(X,X)$ for some set $X = \set{ x_1,\dots,x_n} \subset \Omega$ it is easy to see that the identity
\begin{equation}
 k(X,X) = \sum\limits_{i=1}^p k_i(X,X) \kron Q_i
\end{equation}
holds, where $ k_i(X,X) \kron Q_i$ denotes the Kronecker product. Since sums and Kronecker products of positive (semi-)definite matrices are positive (semi-)definite, we can conclude
that the positive definiteness of $k_i$ and positive (semi-)definiteness of $Q_i$ is sufficient to guarantee that the kernel $k$ is  positive definite. In order to
guarantee strict positive definiteness of $k$ further assumptions on $k_i$ and $Q_i$ have to be made:

\begin{lemma}[Separable kernel is s.p.d] \label{lem: Separable kernel is s.p.d}
 Let $k : \Omega \times \Omega  \rightarrow \R^{m \times m}$ be a separable kernel of order $p$ with decomposition $(k_i,Q_i)_{i=1}^{p}$.
 If the kernels $k_1,\dots,k_p$ are s.p.d. and the matrices $Q_1,\dots,Q_p \succeq 0$ are positive semi-definite, such that $\sum_{i=1}^p Q_i \succ 0$ 
 is positive definite, then $k$ is s.p.d.
\end{lemma}

\begin{proof}
 Let $X = \set{ x_1,\dots x_n} \subset \Omega$ be a set of pairwise distinct points.
 Furthermore, let $K_i := k_i(X,X) \in \R^{n\times n}$ and $K := k(X,X) \in \R^{mn \times mn}$. It holds
 \begin{equation*}
  K = \sum\limits_{i = 1}^{p} K_i \kron Q_i.
 \end{equation*}
 Since the kernels $k_i$, $i= 1,\dots,p$ are s.p.d the matrices $K_1,\dots,K_p$ are positive definite which implies $K_i \succeq \lambda I_{n}$, where
 \begin{equation*}
  \lambda = \min\set{ \lambda^{\prime}  | \, \lambda^{\prime} \text{ is an eigenvalue of } K_i \text{ for some } i \in \set{1,\dots,p} } > 0.
 \end{equation*}
 Therefore,
 \begin{equation*}
  K = \sum\limits_{i = 1}^{p} K_i \kron Q_i \succeq \sum\limits_{i=1}^{p} \lambda I_n \kron Q_i = 
  \lambda I_{n} \kron \left( \sum\limits_{i=1}^{p} Q_i\right) \succ 0.
 \end{equation*}
\end{proof}
It is worthwhile to mention that the assumption $\sum_{i=1}^p Q_i \succ 0$ also guarantees that the kernel $k$ is universal, c.f \cite{Steinwart2001,MXZ2006},
if the scalar-valued kernels $k_i$ are universal. This means that for every compact subset $\Omega_{c} \subset \Omega$ the space
\begin{equation*}
 \myspan \set{ k(\cdot,x)\alpha | \, x \in \Omega_{c}, \, \alpha \in \R^m }
\end{equation*}
is dense in the set of continuous function $C(\Omega_{c})$ over $\Omega_{c}$. For a proof we refer to \cite{CMPY2008}.

In the above case of Definition \ref{def: Separable Kernels} we call $p$ minimal if any other decomposition of $k$ has at least length $p$. This minimality can be directly related to the linear independency
of the set of scalar-valued kernels $\set{k_i}_{i=1}^{p}$ and symmetric matrices $\set{ Q_i}_{i=1}^p$:

\begin{lemma}[Sufficient and necessary minimality condition] \label{lem: Sufficient and necessary minimality condition}
Let $k$ be a separable kernel such that there exists a decomposition of length $p$. Then the following properties are equivalent
\begin{itemize}
 \item[i)] $p$ is minimal.
 \item[ii)] For any decomposition $(k_i,Q_i)_{i=1}^p$ of length $p$ the sets $\set{k_1,\dots,k_p}$ and $\set{Q_1,\dots,Q_p}$ are linearly independent, respectively.
\end{itemize}
\end{lemma}

\begin{proof}
 ``$\Rightarrow$'' Let $(k_i,Q_i)_{i=1}^p$ be a decomposition of length $p$. Assume that either $\set{k_1,\dots,k_p}$ or $\set{Q_1,\dots,Q_p}$ is linearly dependent. W.l.o.g. we can assume that either
 \begin{equation*}
  k_1 = \sum\limits_{i=2}^p \alpha_i k_i \quad \text{ or } \quad Q_1 = \sum\limits_{i=2}^p \beta_i Q_i.
 \end{equation*}
Therefore,
\begin{equation*}
 k = \sum\limits_{i=1}^p k_iQ_i = \sum\limits_{i=2}^p k_i(Q_i +  \alpha_iQ_1) \quad \text{ or } \quad k = \sum\limits_{i=2}^p (k_i + \beta_ik_1)Q_i.
\end{equation*}
In either case we found a smaller decomposition which contradicts the minimality of $p$ since $Q_i + \alpha_iQ_1$ is still symmetric and $k_i + \beta_ik_1$ is still a m.v.-kernel.\\
``$\Leftarrow$'' Let $(k_i,Q_i)_{i=1}^p$ be a decomposition of length $p$ such that $\set{k_1,\dots,k_p}$ and $\set{Q_1,\dots,Q_p}$ are linearly independent.
Assume there exists a decomposition $(\hat{k}_i,\hat{Q}_i)_{i=1}^{q}$ of length $q < p$. Let $\myvec : \R^{m\times m} \rightarrow \R^{m^2}$ denote the vectorization operator. We have
\begin{equation*}
 \sum\limits_{i=1}^p k_iQ_i = k = \sum\limits_{j=1}^q \hat{k}_j\hat{Q}_j
\end{equation*}
and thus
\begin{equation*}
 \sum\limits_{i=1}^p \myvec(Q_i)k_i = \sum\limits_{j=1}^q \myvec(\hat{Q}_j) \hat{k}_j.
\end{equation*}
Setting $Q := [\myvec(Q_1) \dots \myvec(Q_p)] \in \R^{m^2 \times p}$, $\hat{Q} := [\myvec(\hat{Q}_1) \dots \myvec(\hat{Q}_q)] \in \R^{m^2 \times q}$ we get
\begin{equation}
 Q \begin{pmatrix}
    k_1 \\ \vdots \\ k_p 
   \end{pmatrix}
= \hat{Q} \begin{pmatrix}
           \hat{k}_1 \\ \vdots \\ \hat{k}_q
          \end{pmatrix}. \label{eq: vectorization}
\end{equation}
Since $\set{Q_1,\dots,Q_p}$ is linearly indendent it holds $\rank(Q) = p$ and, therefore, there exists a left inverse $A \in \R^{p \times m^2}$, i.e.\, $AQ = I_p$.
Multiplying both sides in \eqref{eq: vectorization} with $A$ from the left, we get
\begin{equation*}
 \begin{pmatrix}
  k_1 \\ \vdots \\ k_p
 \end{pmatrix}
 = A\hat{Q}
 \begin{pmatrix}
  \hat{k}_1 \\ \vdots \\ \hat{k}_q
 \end{pmatrix}
  = \hat{A}
  \begin{pmatrix}
   \hat{k}_1 \\ \vdots \\ \hat{k}_q
  \end{pmatrix}
\end{equation*}
with $\hat{A} := A\hat{Q} \in \R^{p \times q}$. Ultimately, we get $\myspan\set{k_1,\dots,k_p} \subset \myspan\set{\hat{k}_1,\dots,\hat{k}_q}$ which contradicts the linear independency 
of $\set{k_1,\dots,k_p}$.
\end{proof}

It is clear that $k$ given by \eqref{def: separable kernel} is a matrix-valued kernel with regards to Definition \ref{def: Matrix-valued kernel}, as
\begin{equation*}
 k(x,y)^T = \left(\sum\limits_{i=1}^{p} k_i(x,y)Q_i\right)^T = \sum\limits_{i=1}^{p} k_i(x,y)Q_i^T = \sum\limits_{i=1}^{p} k_i(y,x)Q_i = k(y,x).
\end{equation*}
However, the minimality of $p$ by no means implies the uniqueness of the decomposition in the sense that for decompositions $(k_i,Q_i)_{i=1}^p$ and $( \hat{k}_i,\hat{Q}_i)_{i=1}^p$ there exists a permutation
$\iota$ of $ \{ 1, \dots, p \}$ such that 
\begin{equation*}
 k_iQ_i = \hat{k}_{\iota(i)}\hat{Q}_{\iota(i)}.
\end{equation*}
This is illustrated by the following example:
\begin{example}
 Let $k_1,k_2 : \Omega \times \Omega \rightarrow \R$ denote two linearly independent scalar-valued kernels. Then $k : \Omega \times \Omega \rightarrow \R^2$ given by
 \begin{equation*}
  k(x,y) := \begin{pmatrix}
        k_1(x,y) & 0 \\ 0 & k_1(x,y) + k_2(x,y)
       \end{pmatrix}
 \end{equation*}
  denotes a matrix-valued-kernel which has infinitely many minimal decompositions. Let $\lambda \in \R$, then
\begin{equation*}
k(x,y) = k_1(x,y) Q_1(\lambda)
+ ((1-\lambda)k_1+k_2)(x,y) Q_2,
 \end{equation*} 
 where
\begin{equation*}
 Q_1(\lambda) = \begin{pmatrix}
       1 & 0 \\ 0 & \lambda
      \end{pmatrix}, \quad
 Q_2 = \begin{pmatrix}
             0 & 0 \\ 0 & 1
            \end{pmatrix}.
\end{equation*}
\end{example}
We note that there exists only one decomposition for which the spaces spanned by the columns of $Q_1(\lambda)$ and $Q_2$ have zero intersection. This leads us 
to the definition of a subclass of separable kernels:

\begin{definition}[Uncoupled separable kernels] \label{def: Uncoupled separable Kernels}
Let $k : \Omega \times \Omega \rightarrow \R^{m \times m}$ be a separable matrix-valued kernel and $(k_i,Q_i)_{i=1}^p$ be a decomposition.
The decomposition is called \colemph{uncoupled} if
\begin{equation}
 \rank\left(\sum\limits_{i=1}^p Q_i\right) = \sum\limits_{i=1}^p \rank(Q_i) \label{def:eq: Uncoupled condition}
 \end{equation}
If there exists at least one uncoupled decomposition, the kernel is also called \colemph{uncoupled}.
\end{definition}
Using the abbreviation $Q := \sum\limits_{i=1}^p Q_i$, the rank condition \eqref{def:eq: Uncoupled condition} is equivalent to the assumption that the range
$R(Q) := \myspan \set{ Q\alpha | \alpha \in \R^m }$ is equal to the direct sum of the ranges $R(Q_i)$ of the individual matrices. We will state this in the following
Lemma:

\begin{lemma} \label{lem: Sum of Ranges}
 Let $Q_1. \dots, Q_p \in \R^{m \times m}$ be symmetric matrices. Then the following statements are equivalent
 \begin{itemize}
  \item[i)] $ \rank\left(\sum\limits_{i=1}^p Q_i\right) = \sum\limits_{i=1}^p \rank(Q_i) $
  \item[ii)] $ R\left( \sum\limits_{i=1}^p Q_i\right) = \bigoplus\limits_{i=1}^p R(Q_i)$.
 \end{itemize}
\end{lemma}

\begin{proof}
 ``$\Rightarrow$'' We first show that the sum is direct. It is sufficient to show that $R(Q_i) \cap R(Q_j) = \set{0}$ for $i \neq j$. W.l.o.g. we assume $i=1$ and $j=2$ and 
 $R(Q_1) \cap R(Q_2) \neq \set{0}$. It follows $\dim( R(Q_1) \cap R(Q_1) ) \geq 1$ and thus
 \begin{align*}
  \rank\left(\sum\limits_{i=1}^p Q_i\right) & = \dim R\left(\sum\limits_{i=1}^p Q_i\right) \\
  & \leq \dim R(Q_1) + \dim R(Q_2) - \dim(R(Q_1)\cap R(Q_2)) + \dim R\left(\sum\limits_{i=3}^p Q_i\right) \\
  & < \sum\limits_{i=1}^p \dim R(Q_i) = \sum\limits_{i=1}^p \rank(Q_i),
 \end{align*}
 which contradicts $i)$. It is obvious that $R\left( \sum\limits_{i=1}^p Q_i\right) \subset \bigoplus\limits_{i=1}^p R(Q_i)$ and by $i)$ the vector spaces have the 
 same dimension and are therefore equal. \\
 ``$\Leftarrow$'' Since the sum is direct it holds
 \begin{equation*}
  \rank\left(\sum\limits_{i=1}^p Q_i\right) = \dim R\left(\sum\limits_{i=1}^p Q_i\right) = \dim\left(\bigoplus\limits_{i=1}^p R(Q_i)\right) = \sum\limits_{i=1}^p \dim R(Q_i)
  = \sum\limits_{i=1}^p \rank(Q_i).
 \end{equation*}
\end{proof}
With the notion of uncoupledness we can now impose a sufficient condition for the uniqueness of a minimal decomposition up to permutations and scalings:

\begin{theorem}[Uniqueness of uncoupled decompositions] \label{thm: Uniqueness of uncoupled decomposition}
Let $k$ be a separable matrix-valued kernel with uncoupled decomposition $(k_i,Q_i)_{i=1}^p$. If $p$ is minimal, then the decomposition is unique,
up to permutations and scalings.
\end{theorem}
\begin{proof}
Since the decomposition is uncoupled we have $R(Q_1) \cap R(Q_2 + \dots + Q_p) = \set{0}$ and
$R(Q_1 + Q_2 + \dots + Q_p) = R(Q_1) + R(Q_2 + \dots + Q_p)$. Therefore, there exists a $c \in \R^{m}$ such that $Q_1c \neq 0$ and $Q_2c, \dots, Q_p c = 0$. We get
\begin{align*}
k_1 \underbrace{Q_1c}_{\neq 0} & = k_1 Q_1c + k_2Q_2c + \dots + k_pQ_pc \\
& = \hat{k}_1 \hat{Q}_1c + \dots \hat{k}_p \hat{Q}_pc.
\end{align*}
Thus, $k_1$ can be written as a linear combination of $\hat{k}_1,\dots,\hat{k}_p$ such that
\begin{equation*}
k_1 = \sum\limits_{i=1}^p \hat{k}_i a_{1i}.
\end{equation*}
Similarly, the same holds for $k_2,\dots,k_p$ and therefore there exists a matrix $A = (a_{ij})_{i,j=1}^p$ such that
\begin{equation*}
\begin{pmatrix}
k_1 \\ \vdots \\ k_p
\end{pmatrix}
= A \begin{pmatrix}
\hat{k}_1 \\ \vdots \\ \hat{k}_p
\end{pmatrix}.
\end{equation*}
Furthermore, it holds for $i=1,\dots,p$:
\begin{equation*}
 \hat{Q}_j = \sum\limits_{i=1}^p a_{i,j}Q_i  
\end{equation*}
and therefore 
\begin{equation*}
 R(\hat{Q}_j) = \bigoplus\limits_{i=1}^p a_{i,j} R(Q_i).
\end{equation*}
Since $(\hat{k}_j,\hat{Q}_j)_{j=1}^p$ is uncoupled it holds for $j \neq j^{\prime}$:
\begin{equation*}
 R(\hat{Q}_j) \cap R(\hat{Q}_{j^{\prime}}) = \set{0},
\end{equation*}
from which we conclude that $a_{i,j}$ or $a_{i,j^{\prime}}$ is equal to $0$. Thus, for every $i$ there is exactly one $j = j(i)$ such that $a_{i,j(i)} \neq 0$
and the mapping $i \mapsto j(i)$ is bijective and it holds
\begin{equation*}
 k_i = \sum\limits_{j=1}^p a_{i,j}\hat{k}_j = a_{i,j(i)} \hat{k}_{j(i)}.
\end{equation*}
Since both $(k_i,Q_i)_{i=1}^p$ and $(\hat{k}_j,\hat{Q}_j)_{j=1}^p$ are decompositions of $k$ we get
\begin{align*}
 0 = k - k & = \sum\limits_{i=1}^p k_iQ_i - \sum\limits_{j=1}^p \hat{k}_j \hat{Q}_j \\
 & = \sum\limits_{i=1}^p k_iQ_i - \sum\limits_{i=1}^p \hat{k}_{j(i)} \hat{Q}_{j(i)} \\
 & = \sum\limits_{i=1}^p \hat{k}_{j(i)}( a_{i,j(i)} Q_i - \hat{Q}_{j(i)}).
\end{align*}
Since the kernels $k_1,\dots,k_p$ and $\hat{k}_1,\dots,\hat{k}_p$ are linearly independent by Lemma \ref{lem: Sufficient and necessary minimality condition}, respectively,
we conclude that
\begin{equation*}
 a_{i,j(i)}Q_i = \hat{Q}_{j(i)}
\end{equation*}
which results in
\begin{equation*}
 k_iQ_i = \hat{k}_{j(i)} a_{i,j(i)}Q_i = \hat{k}_{j(i)}\hat{Q}_{j(i)}.
\end{equation*}
\end{proof}
In general the existence of an uncoupled or even minimal uncoupled decomposition cannot be guaranteed, as \eqref{def:eq: Uncoupled condition} necessitates that the length of
any uncoupled decomposition is at most $m$. Therefore, any separable kernel of order $m+1$ possesses no uncoupled decomposition. In the following we want to present a sufficient
criterion for the existence of a minimal uncoupled decomposition. This is motivated by trying to extend the well known fact for scalar-valued kernels that the product of two
positive definite kernels is again a positive definite kernel, see \cite{Scholkopf2002a}. This result does not extend to the matrix-valued case, since the kernels additionally have to commute
for every pair of input parameters, i.e.\ for $k_1,k_2 : \Omega \times \Omega \rightarrow \R^{m \times m}$ it must hold
\begin{equation}
 k_1(x,y)k_2(x,y) = k_2(x,y)k_1(x,y), \quad \forall \,  (x,y) \in \Omega \times \Omega \label{eq: kernels commute}
\end{equation}
to have that $k := k_1 \cdot k_2$ is a matrix-valued kernel. 
However, even if \eqref{eq: kernels commute} is satisfied and both $k_1,k_2$ are positive definite the kernel $k$ can be indefinite, as the following example shows:

\begin{example}
 Let $k_1,k_2 : \R \times \R \rightarrow \R$ be given by
 \begin{equation*}
  k_1(x,y) := e^{ -\frac{1}{10}(x-y)^2} \quad \text{ and } \quad k_2(x,y) := e^{-(x-y)^2}
 \end{equation*}
and let $Q_1,Q_2 \in \R^{2 \times 2}$ be the symmetric matrices
\begin{equation*}
 Q_1 = \begin{pmatrix}
        1 & 1 \\ 1 & 1
       \end{pmatrix}
\quad \text{ and } \quad
 Q_2 = \begin{pmatrix}
       0 & 0 \\ 0 & 1
      \end{pmatrix}.
\end{equation*}
Furthermore, let $k : \Omega \times \Omega \rightarrow \R^{2 \times 2}$ denote the matrix-valued kernel with decomposition $(k_i,Q_i)_{i=1}^2$ and $X = \set{0,1}$. 
By Lemma \ref{lem: Separable kernel is s.p.d} $k$ is a positive definite kernel, but $k^2$ is not, as
\begin{equation*}
 k^2(X,X) = \begin{pmatrix}
                     5 & 3 & 2e^{-\frac{1}{5}} + 2e^{-\frac{11}{10}} +e^{-2} & 2e^{-\frac{1}{5}} + e^{-\frac{11}{10}} \\
                     3 & 2 & 2e^{-\frac{1}{5}} + e^{-\frac{11}{10}} & 2e^{-\frac{1}{5}} \\
                     2e^{-\frac{1}{5}} + 2e^{-\frac{11}{10}} +e^{-2} & 2e^{-\frac{1}{5}} + e^{-\frac{11}{10}} & 5 & 3 \\
                     2e^{-\frac{1}{5}} + e^{-\frac{11}{10}} & 2e^{-\frac{1}{5}} & 3 & 2 \\
                    \end{pmatrix}
 \end{equation*}
 has a negative eigenvalue $\lambda \approx -0.044$.
\end{example}
Taking a closer look, the matrix $k^2(X,X)$ can be written as a block-Hadamard product 
\begin{equation*}
 k^2(X,X) = k(X,X) \, \square \, k(X,X) := (k(x_i,x_j)k(x_i,x_j))_{i,j}.
\end{equation*}
As it was shown in \cite{GK2012}, the block-Hadamard product of two positive (semi-)definite block matrices $A = (A_{ij})_{i,j}$, $B = (B_{ij})_{i,j}$ is positive (semi-)definite
if each block of $A$ commutes with each block of $B$. If this restriction is applied to every possible Gramian matrix of a matrix-valued kernel, this leads to the condition
\begin{equation*}
 k(x,y)k(\tilde{x},\tilde{y}) =  k(\tilde{x},\tilde{y})k(x,y), \quad \forall \, x,y,\tilde{x},\tilde{y} \in \Omega.
\end{equation*}
In this case, the kernel $k$ can be characterized as follows:

\begin{theorem} \label{thm: characterization of orthogonal kernel}
 Let $k : \Omega \times \Omega \rightarrow \R^{m \times m}$ be matrix-valued kernel such that $k(x,y) = k(y,x)$ for all $x,y \in \Omega$. Then the following statements are equivalent
 \begin{itemize}
  \item[i)] $k(x,y)k(\tilde{x},\tilde{y}) = k(\tilde{x},\tilde{y})k(x,y)$ for all $x,\tilde{x},y,\tilde{y} \in \Omega$
  \item[ii)] There exists an orthogonal matrix $P \in \R^{m \times m}$ such that $P^Tk(x,y)P$ is diagonal for all $x,y \in \Omega$.
  \item[iii)] $k$ is separable and there exists an uncoupled decomposition $(k_i,Q_i)_{i=1}^p$ with length $p \leq m$ and for which $Q_iQ_j = 0$ for $i \neq j$.
 \end{itemize}
\end{theorem}

\begin{proof}
 ``$i) \Rightarrow ii)$'' Let $A_1,\dots,A_d$ denote a basis of $\myspan\set{ k(x,y) | x,y \in \Omega}$.
 Then the $A_i$ are symmetric, commute with one other and therefore are simultaneously diagonalizable, i.e.\,
 there exists an orthogonal matrix $P$ such that $P^TA_iP$ is diagonal for $i = 1,\dots,m$. It follows,
 that $k(x,y) \in \myspan\set{ A_1,\dots A_d}$ is diagonalizable for any $x,y \in \Omega$\\
 ``$ii) \Rightarrow iii)$'' By assumption it holds
 \begin{equation*}
  P^Tk(x,y)P = \diag( k_1(x,y), \dots, k_d(x,y))
 \end{equation*}
 and $k_i : \Omega \times \Omega \rightarrow \R$, $i=1,\dots,d$ are scalar-valued kernels.
 For $i = 1,\dots,d$ let $J(i) := \set{ j : k_i = \alpha_{i,j} k_j \text{ for some } \alpha_{i,j} \in \R}$. Then there exist $i_1,\dots,i_p$ with minimal $p$ such that 
 \begin{equation*}
  \bigcup\limits_{l=1}^p J(i_l) = \set{1,\dots,d} \quad \text{ and } \quad J(i) \cap J(i^{\prime}) = \emptyset \text{ for } i \neq i^{\prime}.
 \end{equation*}
 It holds
 \begin{align*}
  k &= \sum\limits_{i=1}^m k_i (Pe_i)(Pe_i)^T = \sum\limits_{l=1}^m k_{i_l} \underbrace{\sum\limits_{j \in J(i_l)} \alpha_{i_l,j} (Pe_j)(Pe_j)^T}_{=: Q_{i_l}} = \sum\limits_{l=1}^m k_{i_l} Q_{i_l}.
 \end{align*}
 Furthermore,
 \begin{equation*}
  Q_iQ_{i^{\prime}} = \sum\limits_{j \in J(i)} \sum\limits_{j^{\prime} \in J(i^{\prime})} \alpha_{i,j}\alpha_{i^{\prime},j^{\prime}} (Pe_j)\underbrace{(Pe_j)^T(Pe_{j^{\prime}})}_{= 0}(Pe_{j^{\prime}})^T = 0
 \end{equation*}

``$iii) \Rightarrow i)$'' It holds
\begin{align*}
 k(x,y)k(\tilde{x},\tilde{y}) & = \left(\sum\limits_{i=1}^p k_i(x,y)Q_i \right)\left( \sum\limits_{j=1}^p k_j(\tilde{x},\tilde{y}) Q_j \right)\\
 & = \sum\limits_{i=1}^p \sum\limits_{j=1}^p k_i(x,y)k_j(\tilde{x},\tilde{y})Q_iQ_j \\
 & = \sum\limits_{i=1}^p k_i(x,y)k_i(\tilde{x},\tilde{y}) Q_i^2 \\
 & = \sum\limits_{i=1}^p k_i(\tilde{x},\tilde{y})k_i(x,y) Q_i^2 \\
 & = \sum\limits_{i=1}^p \sum\limits_{j=1}^p k_j(\tilde{x},\tilde{y})k_i(x,y)Q_jQ_i \\
 & = \left( \sum\limits_{j=1}^p k_j(\tilde{x},\tilde{y}) Q_j \right)\left(\sum\limits_{i=1}^p k_i(x,y)Q_i \right)\\
 & = k(\tilde{x},\tilde{y})k(x,y).
\end{align*}
\end{proof}
We conclude this subsection with a direct corollary:

\begin{corollary}
 Let $k : \Omega \times \Omega \rightarrow \R^{m \times m}$ be a positive definite matrix-valued kernel that satisfies $k(x,y) = k(y,x)$ for all $x,y \in \Omega$. If one of the conditions in Theorem
 \ref{thm: characterization of orthogonal kernel} is met, then $k^n$ is a positive definite matrix-valued kernel for any $n \in \N_0$.
\end{corollary}

\begin{proof}
 By Theorem \ref{thm: characterization of orthogonal kernel} $k$ can be decomposed as
 \begin{equation*}
  k(x,y) = \sum\limits_{l=1}^p k_lQ_l
 \end{equation*}
 with positive-definite scalar-valued kernels $k_l$ and positive semi-definite matrices $Q_l$ satisfying $Q_lQ_{l^{\prime}} = 0$ for $l \neq l^{\prime}$. Therefore,
 for any set $X = \set{x_1,\dots,x_n}$ of p.w.\ distinct points
 \begin{equation*}
  k^{n}(X,X) = \sum\limits_{l=1}^p \underbrace{\underbrace{k_l^{n}(X,X)}_{\succ 0} \kron \underbrace{Q_l}_{\succeq 0}}_{\succeq 0} \succeq 0.
 \end{equation*}

\end{proof}

\subsection{RKHS for separable kernels}

As we want to consider approximations in the RKHS of separable kernels, we will show how the RKHS of the matrix-valued kernel $k$ relates to the RKHS of the scalar-valued
kernels $k_i$ and matrices $Q_i$ which form a decomposition of $k$. We start with decompositions of order $1$:

\begin{lemma}[RKHS of separable kernels of order $1$] \label{lem: RKHS separable kernel order 1}
 Let $k_s$ be a scalar-valued p.d. kernel and $Q \in \R^{m \times m}$ a positive semi-definite matrix.
 Then $k := k_s\cdot Q$ is a p.d. matrix-valued kernel and it holds
 \begin{equation}
  \cH_k = \cH_{k_s}e_1 \oplus \dots \oplus \cH_{k_s}e_p. \label{eq: RKHS for separable kernel of order 1}
 \end{equation}
 Here $\set{ e_i}_{i=1}^p$ denotes a basis of the range of $Q$.
\end{lemma}
\begin{proof}
 We first show that the sum is direct. Let $f_i \in \cH_{k_s}e_i$, $i = 1,..,p$. Assume that
 \begin{equation*}
  f_1 + \dots + f_p = 0
 \end{equation*}
and there is at least one $j \in \set{1,\dots,p}$ such that $f_j \neq 0$. It follows
\begin{equation*}
 \myspan\set{e_j} \ni f_j(x) = - \sum\limits_{i=1,i\neq j}^p f_i(x) \in \myspan\set{e_1,\dots,e_{j-1},e_{j+1},\dots,e_{p} }, \quad \forall \, x \in \Omega
\end{equation*}
and therefore
\begin{equation*}
 f_j(x) \in \myspan\set{e_j} \cap \myspan\set{e_1,\dots,e_{j-1},e_{j+1},\dots,e_{p} } = \set{0}, \quad \forall \, x \in \Omega,
\end{equation*}
i.e.\ $f_j = 0$. Iteratively we get $f_i = 0$ for $i = 1,\dots,p$ and the sum is direct. We now show that the right hand side of \eqref{eq: RKHS for separable kernel of order 1}
is a subspace of the left hand side. Therefore, let $f_i \in \cH_{k_s}e_i$, $i = 1,..,p$. Then there exist sequences $(\alpha^{(i)}_n)_{n \in \N} \subset \R$
and $(x^{(i)}_n)_{n \in \N} \subset \Omega$ such that
\begin{align*}
 f_i & = \left(\sum\limits_{n=1}^{\infty} k_s(\cdot,x^{(i)}_n)\alpha^{(i)}_n\right)e_i \\
     & = \sum\limits_{n=1}^{\infty} k(\cdot,x^{(i)}_n) \alpha^{(i)}_nv_i,
\end{align*}
where $v_i \in \R^m$ satisfies $Qv_i = e_i$. We conclude that $f_i \in \cH_{k}$ for $i = 1,\dots,p$ and thus 
\begin{equation*}
 f_1 + \dots + f_p \in \cH_{k}.
\end{equation*}
Assume that $\cH_{k} \neq \cH_{k_s}e_1 \oplus \dots \oplus \cH_{k_s}e_p.$ Then for any 
$f \in \left(\cH_{k_s}e_1 \oplus \dots \oplus \cH_{k_s}e_p.\right)^{\perp}$ it holds
\begin{equation*}
 \Hscalprod{f}{k_s(\cdot,x)e_i} = 0 \quad \forall \, x \in \Omega, i = 1,\dots,p.
\end{equation*}
Due to the linearity of the inner products it also holds
\begin{equation*}
 f(x)^T\alpha = \Hscalprod{f}{k(\cdot,x)\alpha} = 0 \quad \forall \, x \in \Omega
\end{equation*}
and thus $f = 0$.
\end{proof}
\begin{remark}
 In the special case of $Q = I_m$, which where for example considered in \cite{HS2017a,Wirtz2013} this leads to the RKHS $\cH_{k} = \bigotimes\limits_{i=1}^m \cH_{k_s}$ with the inner product given by
 \begin{equation*}
  \langle f , g \rangle_{\cH_k} = \langle (f_1,\dots,f_m), (g_1,\dots,g_m) \rangle = \sum\limits_{i=1}^m \langle f_i,g_i \rangle_{\cH_{k_s}}.
 \end{equation*}
\end{remark}

We have seen, c.f.\ Corollary \ref{cor: Bound on the interpolation error} that the power-function is a valuable tool to provide error estimators to the pointwise error
between a function $f$ in $\cH$ and its interpolant in a subspace $\cN$. For scalar-valued kernels bounds on the decay of the power-functions are known for a wide variety of kernels,
see \cite{Wendland2005} for more details. We want to make use of these bounds, to derive similar bound for the matrix-valued case. Again, we restrict ourself to the separable
kernels of order $1$ at first:

\begin{lemma}[power-function of separable kernels of order $1$] \label{cor: Power Function of separable kernel of order 1}
 Let $k: \Omega \times \Omega \rightarrow \R^{m \times m}$ be a separable kernel of order $1$ with decomposition $(k_1,Q_1)$, 
 where $k_1$ is a p.d. kernel and $Q_1$ is positive semi-definite.
 Let $X_n = \set{ x_1,\dots,x_n} \subset \Omega$ be a set of pairwise distinct points. Furthermore, let $\cN := \cN_{k}(X_n)$,
 $\hat{\cN} := \myspan\set{k_1(\cdot,x_j) | \, x_j \in X_n}$ and let $\cP_{\hat{\cN}}$ denote the power function of the 
 scalar-valued kernel $k_1$. Then it holds
 \begin{equation}
  \left(\cP_{\cN}^{\alpha}(x)\right)^2 = \cP_{\hat{\cN}}(x)^2 \alpha^TQ_1\alpha. \label{eq: Power Function of separable kernel of order 1}
 \end{equation}
\end{lemma}
\begin{proof}
 Since $k = k_1Q_1$ and due to Corollary \ref{cor: Power Function} it is sufficient to show that $k_{\cN} = k_{1,\hat{\cN}}Q_1$.
 Let $K_1:= k_1(X_n,X_n)$. It is easy to see that $k(x,X_n) = k_1(x,X_n) \kron Q_1 \in \R^{mn \times m}$ and $K = k(X_n,X_n) = K_1 \kron Q_1$ and
 therefore by applying Corollary \ref{cor: Reproducing kernel of NX} we get
 \begin{align*}
  k_{\cN}(x,y) &= k(x,X_n)^TK^+k(y,X_n) \\
  & = \left(k_1(x,X_n)\kron Q_1\right)^T \left(K_1 \kron Q_1\right)^{+} k_1(y,X_n) \kron(Q) \\
  & = \left( k_1(x,X_n)^TK_1^{+}k_1(y,X_n)\right)\kron\left(Q_1Q_1^{+}Q_1\right) \\
  & = k_{1,\hat{\cN}}(x,y) \kron Q_1 \\
  & = k_{1,\hat{\cN}}(x,y)  Q_1.
 \end{align*}
\end{proof}

We now extend this result to separable kernels of higher order. It is easy to see that for $k$ with decomposition $(k_i,Q_i)_{i=1}^p$ it holds
\begin{equation}
 \cH = \cH_{1} + \dots + \cH_{p} \label{eq: Sum of RKHS}
\end{equation}
where $\cH_{i}$ denotes the RKHS of the separable kernel $k_iQ_i$ of order $1$. By Lemma \ref{lem: RKHS separable kernel order 1} we know that $\cH_i$ can be written
as a direct sum. However, in \eqref{eq: Sum of RKHS} the sum does no longer need to be direct which causes issues when trying to determine the power-function of $k$ in terms of the
power-function of the kernels $k_i$. This can be traced back to the fact that for a set $X = \set{ x_1,\dots,x_n} \subset \Omega$ the space spanned by the functions $k(\cdot,x)\alpha$ for
$x \in X$ and $\alpha \in \R^m$ is not equal to the sum of the individual subspaces spanned by $k_i(\cdot,x)Q_i\alpha$.

\begin{lemma}[power-function bound of separable kernel of order $p$] \label{lem: power-function bound of separable kernel}
 Let $k : \Omega \times \Omega \rightarrow \R^{m\times m}$ be a separable matrix-valued kernel with decomposition $(k_i,Q_i)_{i=1}^p$ and $X = \set{ x_1,\dots,x_n} \subset \Omega$.
 Furthermore, let $\hat{k}_i := k_iQ_i$ with $\cH_{i}$ as its respective RKHS and
  \begin{align*}
  \cN_i &:= \myspan \set{ \hat{k}_i(\cdot,x)\alpha | \, x \in X, \, \alpha \in \R^m }, \quad i = 1,\dots N, \\
  \cN &:= \myspan \set{ k(\cdot,x)\alpha | \, x \in X, \, \alpha \in \R^m }
 \end{align*}
 Then it holds for all $x \in \Omega $ and $\alpha \in \R^{m}$:
 \begin{equation}
  \sum\limits_{i=1}^p \left( \cP_{\cN_i}^{\alpha}(x)\right)^2 \leq \cP_{\cN}^{\alpha}(x)^2.
 \end{equation}
\end{lemma}

\begin{proof}
  $\cN_{i} \subset \cH_i$ is a closed subspace with reproducing kernel $\hat{k}_{i,\cN_i}$ and by Corollary \ref{cor: Power Function} it holds
 \begin{equation}
  \cP_{\cN_i}^{\alpha}(x)^2 = \alpha^T\left(\hat{k}_i(x,x) - \hat{k}_{i,\cN_i}(x,x)\right)\alpha. \label{eq: power-function proof}
 \end{equation}
We make use of the fact that the sum $\cM := \cN_1 + \dots + \cN_p$ is an RKHS with reproducing kernel $k_{\cM} = \hat{k}_{1,\cN_1} + \dots + \hat{k}_{p,\cN_p}$ and norm given by
\begin{equation*}
 \| f \|_{\cM} = \min \set{ \sum\limits_{i=1}^p \| f_i \|_{\cN_i}^2 \left.\right| \, f = \sum\limits_{i=1}^p f_i, \, f_i \in \cN_i}.
\end{equation*}
A proof for this assertion for the scalar-valued case can be found in \cite{Aronszajn1950}. The proof for the matrix-valued case only involves minor modifications.
For the sake of completeness it is shown in the appendix.

It now holds
\begin{align*}
 \| f \|_{\cM} & = \min \set{ \sum\limits_{i=1}^p \| f_i \|_{\cN_i}^2 \left.\right| \, f = \sum\limits_{i=1}^p f_i, \, f_i \in \cN_i} \\
	       & = \min \set{ \sum\limits_{i=1}^p \| f_i \|_{\cH_i}^2 \left.\right| \, f = \sum\limits_{i=1}^p f_i, \, f_i \in \cN_i} \\
	       & = \| f \|_{\cH}
\end{align*}
and therefore $k_{\cM}$ is the reproducing kernel of the subspace $\cM \subset \cH$. Using \eqref{eq: power-function proof} and Corollary \ref{cor: Power Function} we conclude that
\begin{align}
 \cP_{\cM}^{\alpha}(x)^2 & = \alpha^T\left( \sum\limits_{i=1}^p \hat{k}_i(x,x) - \sum\limits_{i=1}^p \hat{k}_{i,\cN_i}(x,x) \right) \alpha \nonumber \\
 & = \sum\limits_{i=1}^p \alpha^T \left( \hat{k}_i(x,x) - \hat{k}_{i,\cN_i}(x,x)\right) \alpha \nonumber \\
 & = \sum\limits_{i=1}^p \cP_{\cN_i}^{\alpha}(x)^2. \label{eq: power-function sum identity}
\end{align}
Since $\cN \subset \cM$ is a subspace the orthogonal complements satisfy $\cM^{\perp} \subset \cN^{\perp}$ and by applying Corollary \ref{cor: Alternative representation of the power-function} it follows
\begin{equation*}
 \sum\limits_{i=1}^p \cP_{\cN_i}^{\alpha}(x)^2 = \cP_{\cM}^{\alpha}(x)^2 = \Hnorm{ \Pi_{\cM^{\perp}}k(\cdot,x)\alpha} \leq \Hnorm{ \Pi_{\cN^{\perp}}k(\cdot,x)\alpha} = \cP_{\cN}^{\alpha}(x)^2.
\end{equation*}
\end{proof}
We see that in general equality cannot be guaranteed. It only holds if the space $\cM$ is equal to $\cN$. This is equivalent to the fact that all $\hat{k}_{i}(\cdot,x)\alpha$ with $x \in X$ lie in $\cN$.
We will see in the following that this can be achieved when the decomposition is uncoupled:

\begin{lemma}[power-function of uncoupled separable kernels of order $p$]
 Let $k : \Omega \times \Omega \rightarrow \R^{m\times m}$ be a separable matrix-valued kernel with uncoupled decomposition $(k_i,Q_i)_{i=1}^p$ and $X = \set{ x_1,\dots,x_n} \subset \Omega$.
 Furthermore, let $\hat{k}_i := k_iQ_i$ with $\cH_{i}$ as its respective RKHS and
  \begin{align*}
  \cN_i &:= \myspan \set{ \hat{k}_i(\cdot,x)\alpha | \, x \in X, \, \alpha \in \R^m }, \quad i = 1,\dots N, \\
  \cN &:= \myspan \set{ k(\cdot,x)\alpha | \, x \in X, \, \alpha \in \R^m }
 \end{align*}
 Then it holds for all $x \in \Omega $ and $\alpha \in \R^{m}$:
 \begin{equation}
  \sum\limits_{i=1}^p \left( \cP_{\cN_i}^{\alpha}(x)\right)^2 = \cP_{\cN}^{\alpha}(x)^2.
 \end{equation}
\end{lemma}

\begin{proof}
 As mentioned before, it is sufficient to show that $\hat{k}_{i}(\cdot,x)\alpha \in \cN$ for all $x \in X$, $\alpha \in \R^m$ and $i = 1,\dots,p$. Because the decomposition is uncoupled
 it holds with Lemma \ref{lem: Sum of Ranges} that
 \begin{equation*}
  R\left( \sum\limits_{i=1}^p Q_i\right) = \bigoplus\limits_{i=1}^p R(Q_i).
 \end{equation*}
Therefore, for every $\alpha \in \R^m$ there exists a $\beta \in \R^m$ such that $Q_i\alpha = (Q_1 + \dots + Q_p)\beta$. Since the sum is direct it holds that 
$Q_j\beta = 0$ for $j \neq i$ and therefore
\begin{equation*}
 \hat{k}_{i}(\cdot,x)\alpha = k_i(\cdot,x)Q_i\alpha = \sum\limits_{i=1}^p k_i(\cdot,x)Q_i\beta = k(\cdot,x)\beta \in \cN.
\end{equation*}
\end{proof}

Lastly, we want to remark that while a lower bound in terms of the sum of the power-functions for the matrix-valued kernels of order $1$ can be achieved, as seen in
Lemma \ref{lem: power-function bound of separable kernel}, an upper bound of this kind is not available in general as the following example shows.


\begin{example}
 Let $\Omega \subset \R^d$ and $k_1,k_2 : \Omega \times \Omega \rightarrow \R$ be the polynomial kernels given by
 \begin{equation*}
  k_1(x,y) = x^Ty \quad \text{ and } \quad k_2(x,y) = (x^Ty)^2,
 \end{equation*}
 respectively, then the RKHS $\cH_1$ is equal to the space of multivariate polynomials of degree $1$ and $\cH_2$ to the space of multivariate polynomials of degree $2$. In particular,
 $\dim(\cH_1) = d$ and $\dim(\cH_2) = d(d+1)/2$ and therefore by choosing $X = \set{x_i}_{i=1}^{d(d+1)/2}$ such that $\set{k_2(\cdot,x_i)}_{i=1}^{d(d+1)/2}$ is linearly independent, the power-functions
 $\cP_{\cN_1(X)}$ and $\cP_{\cN_2(X)}$ vanish. However, the RKHS for $k := k_1 + k_2$ is given by the space of multivariate polynomials of degree $1$ or $2$ for which $\dim(\cH) = d(d+3)/2$ holds.
 Consequently, $\cN(X) \neq \cH$ and $\cP_{\cN(X)}$ does not vanish.

\end{example}



\section{Numerical Examples} \label{sec: Numerical Example}

\subsection{Example 1}
We now investigate the approximation quality of interpolation with matrix-valued kernels compared to a scalar-valued, i.e.\ componentwise approach. For this, we
consider the target function $f : \Omega := [-2,2] \rightarrow \R^3$ given by
\begin{equation*}
 f(x) := 
 \begin{pmatrix}
        \frac{1}{\sqrt{3}} & \frac{1}{\sqrt{3}} & \frac{1}{\sqrt{3}} \\
        0 		  & \frac{1}{\sqrt{2}} & -\frac{1}{\sqrt{2}} \\
       -\frac{\sqrt{2}}{\sqrt{3}} & \frac{1}{\sqrt{6}} & \frac{1}{\sqrt{6}} \\
 \end{pmatrix}
 \begin{pmatrix}
          e^{-2.5(x-0.5)^2} + e^{-2.0(x+0.5)^2} \\
          e^{-3.5(x-0.7)^2} \\
          1
 \end{pmatrix}
\end{equation*}
and the uncoupled separable kernels $k_1,\dots,k_4 : \Omega \times \Omega \rightarrow \R^{3 \times 3}$ of order $1$, $3$, $2$ and $3$, respectively, given by
\begin{align*}
 k_1(x,y) & := e^{-\varepsilon_{11}(x-y)^2} I_3 \\
 k_2(x,y) & := e^{-\varepsilon_{21}(x-y)^2}e_1e_1^T +  e^{-\varepsilon_{22}(x-y)^2}e_2e_2^T + e^{-\varepsilon_{23}(x-y)^2}e_3e_3^T \\     
 k_3(x,y) & := e^{-\varepsilon_{31}(x-y)^2}v_1v_1^T + e^{-\varepsilon_{32}(x-y)^2}\left(v_2v_2^T + v_3v_3^T\right) \\
 k_4(x,y) & := e^{-\varepsilon_{41}(x-y)^2}v_1v_1^T + e^{-\varepsilon_{42}(x-y)^2}v_2v_2^T + e^{-\varepsilon_{43}(x-y)^2}v_3v_3^T, \\
\end{align*}
with shape parameters $\varepsilon_{11},\dots,\varepsilon_{43} \in (0,\infty)$. 
Here $e_i$ denotes the $i$-th standard basis vector of $\R^3$ and $v_1,v_2,v_3$ are an ONB of eigenvectors of the covariance matrix $C$ of $f$, which is computed by taking $401$ random evaluations of $f$ and setting
\begin{equation*}
  C := \frac{1}{400} \sum\limits_{i=1}^{401} (f_i-\mu)(f_i-\mu)^T,
\end{equation*}
where $\mu \in \R^3$ contains the componentwise mean.

The kernels $k_1$ and $k_2$ handle the data componentwise that is, for the kernel $k_1$ the same scalar-valued kernel is used for every component, 
while for $k_2$ each component is treated by a different scalar-valued kernel. However, for the kernels $k_3$ and $k_4$ this is not the case.
The shape parameters are determined by minimizing the maximum pointwise interpolation error $e_{k_i}(x) := \| f(x) - s_{k_i}(x) \|$ evaluated on a validation set $\Omega_{M}$
of $40$ randomly chosen points in $\Omega$ for $50$  logarithmically equidistantly distributed parameters in $M := [0.1,100]$, where $s_{k_i}$ is the
interpolant on the set of $35$ equidistantly distributed centers $X := \set{ -2 + \frac{4}{34}i | i = 0,\dots,34}$ belonging to the RKHS that corresponds to $k_i$. The resulting parameters are listed in 
Table \ref{table:1}.
\begin{table}[h]
\begin{center}
\begin{tabular}{| c || c || c c c || c c || c c c |}
\hline
Parameter & $\varepsilon_{11}$ &  $\varepsilon_{21}$ &  $\varepsilon_{22}$ &  $\varepsilon_{23}$ 
&  $\varepsilon_{31}$ &  $\varepsilon_{32}$ &  $\varepsilon_{41}$ &  $\varepsilon_{42}$ &  $\varepsilon_{43}$ \\
\hline
Value &  1.931 &  1.931 &  1.931 &  1.600 &  0.244 &  3.393 &  0.244 &  3.393 &  3.393 \\
\hline
\end{tabular}
\caption{Results of the parameter selection for the different kernels.}
\label{table:1}
\end{center}
\end{table}

We note that for the kernels $k_1$ and $k_2$ the selected shape parameters only differ in the third component, 
where a smaller parameter and therefore wider Gaussian was choosen for $k_2$. For the kernels $k_3$ and $k_4$ the
selected parameters result in the same matrix-valued kernel. This can be explained by the fact that the eigenvectors
$v_2$ and $v_3$ of the covariance matrix $C$ were a-priori grouped together based on the fact that their corresponding
eigenvalues $\lambda_2 = 0.112$ and $\lambda_3 = 0.206$ are of similar magnitude. This is reasonable as the eigenvalues
are precisely the standard deviation of the data along the directions $v_2$ and $v_3$ and therefore the same Gaussian might be used
for both directions.

Using the above parameters we compute the maximum pointwise interpolation error $e_{k_i}(X_N)$ on a test set $\Omega_T \subset \Omega$ of $400$ equidistantly distributed points
for an increasing number of equidistant training centers, i.e. $X_N := \set{ -2 + \frac{4}{N-1}i | i = 0,\dots,N-1}$. The results for $N = 1,\dots,35$ are plotted in Figure \ref{figure:1}.

We can see that for a small number of centers, the difference in the approximation quality between the scalar-valued and matrix-valued approach is negligible. However,
as the number of centers $N$ increases, the kernel $k_3 = k_4$ begins to outperform the componentwise kernel $k_1$ and $k_2$. On the one hand, this leads to a higher accuracy for a fixed 
number of centers, i.e.\ a difference of almost three orders of magnitude for $N = 21$. On the other hand, this allows for a smaller expansion size while maintaining the same order of
accuracy and therefore leads to a sparser approximant.

\begin{figure}[ht]
\centering
\includegraphics[scale=0.5]{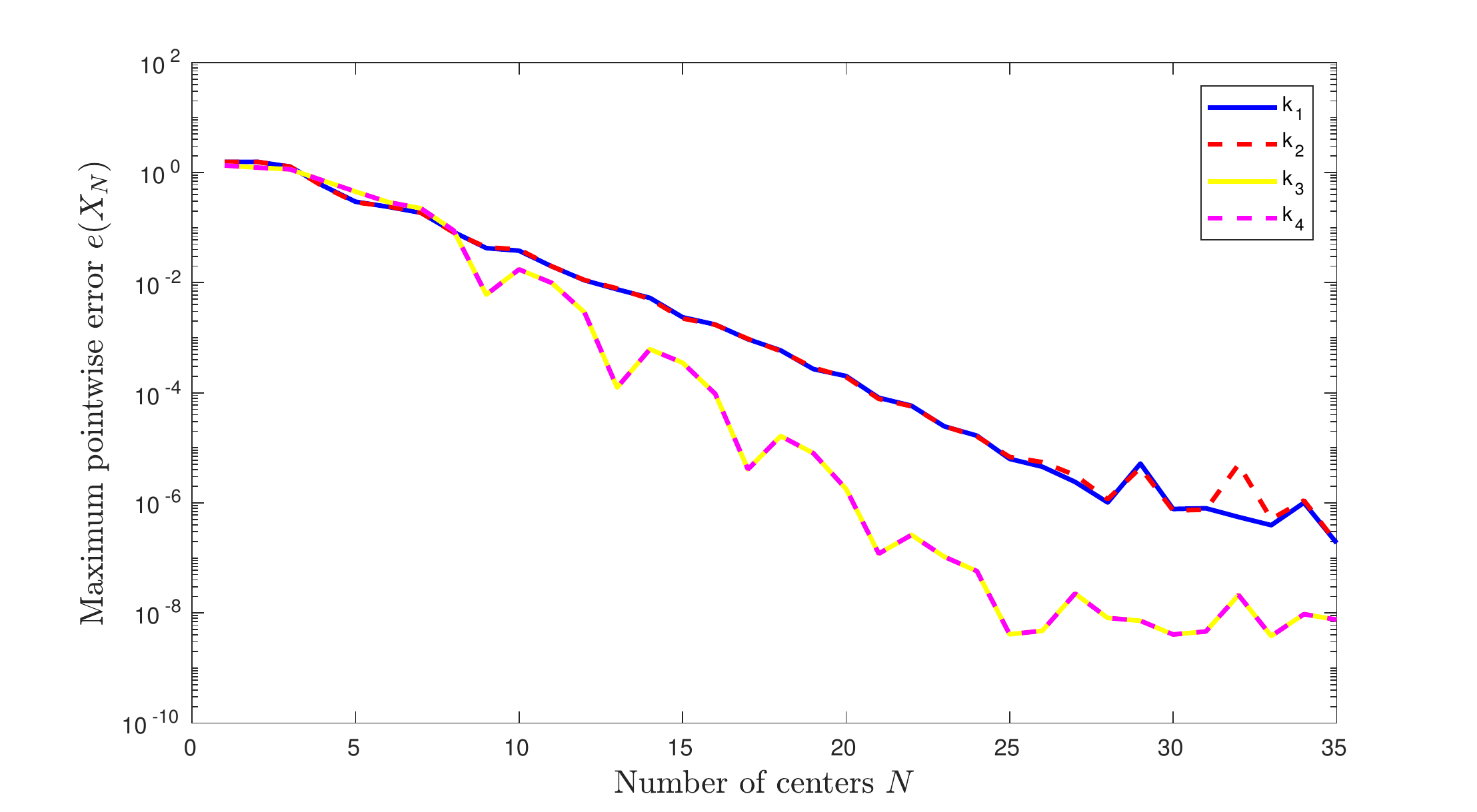}
\caption{Maximum pointwise error measured in the Euclidean norm for the kernels $k_1$ to $k_4$ and for increasing number of centers $N$.}
\label{figure:1}
\end{figure}

\subsection{Example 2}

We now want to verify the validity of the error bounds stated in Corollary \ref{cor: Bound on the interpolation error}. To this end we consider the domain $\Omega := [-1,1]^2$ and the separable
kernel $k$ with decomposition $(k_1,Q_1)_{i=1}^3$ given by $k_i = e^{-i\|x-y\|^2}$ and
\begin{equation*}
 Q_1 = \begin{pmatrix}
       1 & 1 & -1 & -1 \\
       1 & 1 & -1 & -1 \\
      -1 & -1 & 1 & 1  \\
      -1 & -1 & 1 & 1
      \end{pmatrix},
 Q_2 = \begin{pmatrix}
       1 & 0 & 0 & 0 \\
       0 & 1 & 0 & 0 \\
       0 & 0 & 1 & -1  \\
       0 & 0 & -1 & 1
      \end{pmatrix} \text{ and }
 Q_3 = \begin{pmatrix}
       0 & 0 & 0 & 0 \\
       0 & 0 & 0 & 0 \\
       0 & 0 & 1 & -1  \\
       0 & 0 & -1 & 1
      \end{pmatrix}.
\end{equation*}
We consider the target function $f \in \cH$ given by
\begin{equation*}
 f(x) = \sum\limits_{i=1}^5 k(x,y_i)\alpha_i,
\end{equation*}
where $y_1, \dots, y_5 \in \Omega $ and $\alpha_1, \dots, \alpha_5 \in \R^{4}$ were randomly chosen.
We further select $X = \{x_1, \dots, x_{100} \}$ random points and for $X_i := \{x_1, \dots, x_i \}$ compute the error in the Euclidean-, infinity- and one-norm as well as the error bounds
\begin{align*}
 \Delta_{2}^1 := \| k_{\cN(X_i)}(x,x) \|_2 \|f - \Pi_{\cN(X_i)} f \|, & & \Delta_{2}^2 := \| k_{\cN(X_i)}(x,x) \|_2 \|f \|, \\
 \Delta_{\infty}^1 := \max\limits_{j=1,\dots,4} | k_{\cN(X_i)}(x,x)_{jj} | \|f - \Pi_{\cN(X_i)} f \|, & & \Delta_{\infty}^2 := \max\limits_{j=1,\dots,4} | k_{\cN(X_i)}(x,x)_{jj} |\|f \|, \\
 \Delta_{1}^1 := \| k_{\cN(X_i)}(x,x) \|_2 \|f - \Pi_{\cN(X_i)} f \|, & & \Delta_{1}^2 := \| k_{\cN(X_i)}(x,x) \|_2 \|f \| \\
\end{align*}
The results are plotted in Figures \ref{figure:2} - \ref{figure:4}, respectively. 
\begin{figure}[ht]
\centering
\includegraphics[scale=0.5]{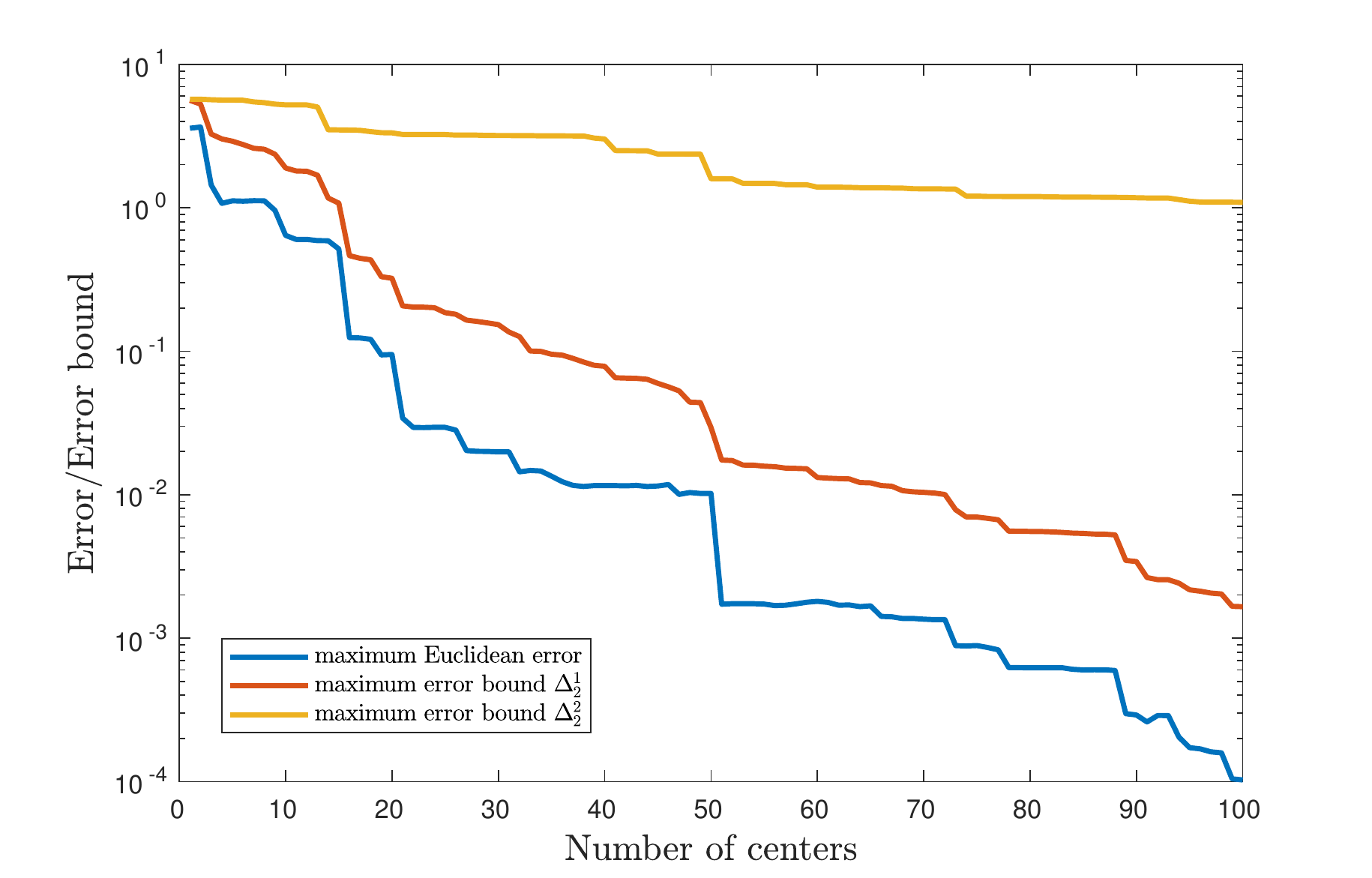}
\caption{Maximum pointwise error measured in the Euclidean norm for increasing number of centers.}
\label{figure:2}
\end{figure}
\begin{figure}[ht]
\centering
\includegraphics[scale=0.5]{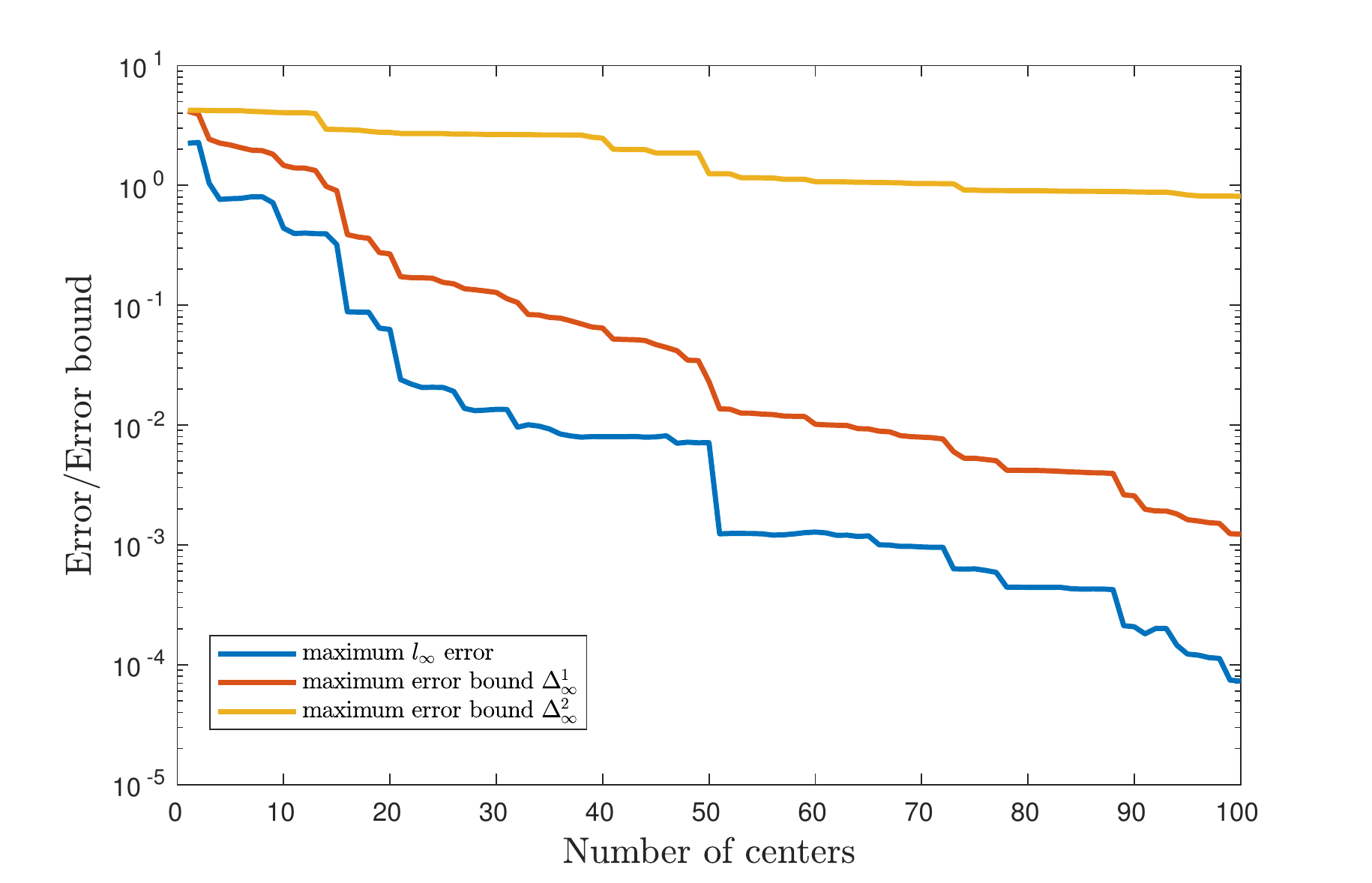}
\caption{Maximum pointwise error measured in the infinity norm for increasing number of centers.}
\label{figure:3}
\end{figure}
\begin{figure}[ht]
\centering
\includegraphics[scale=0.5]{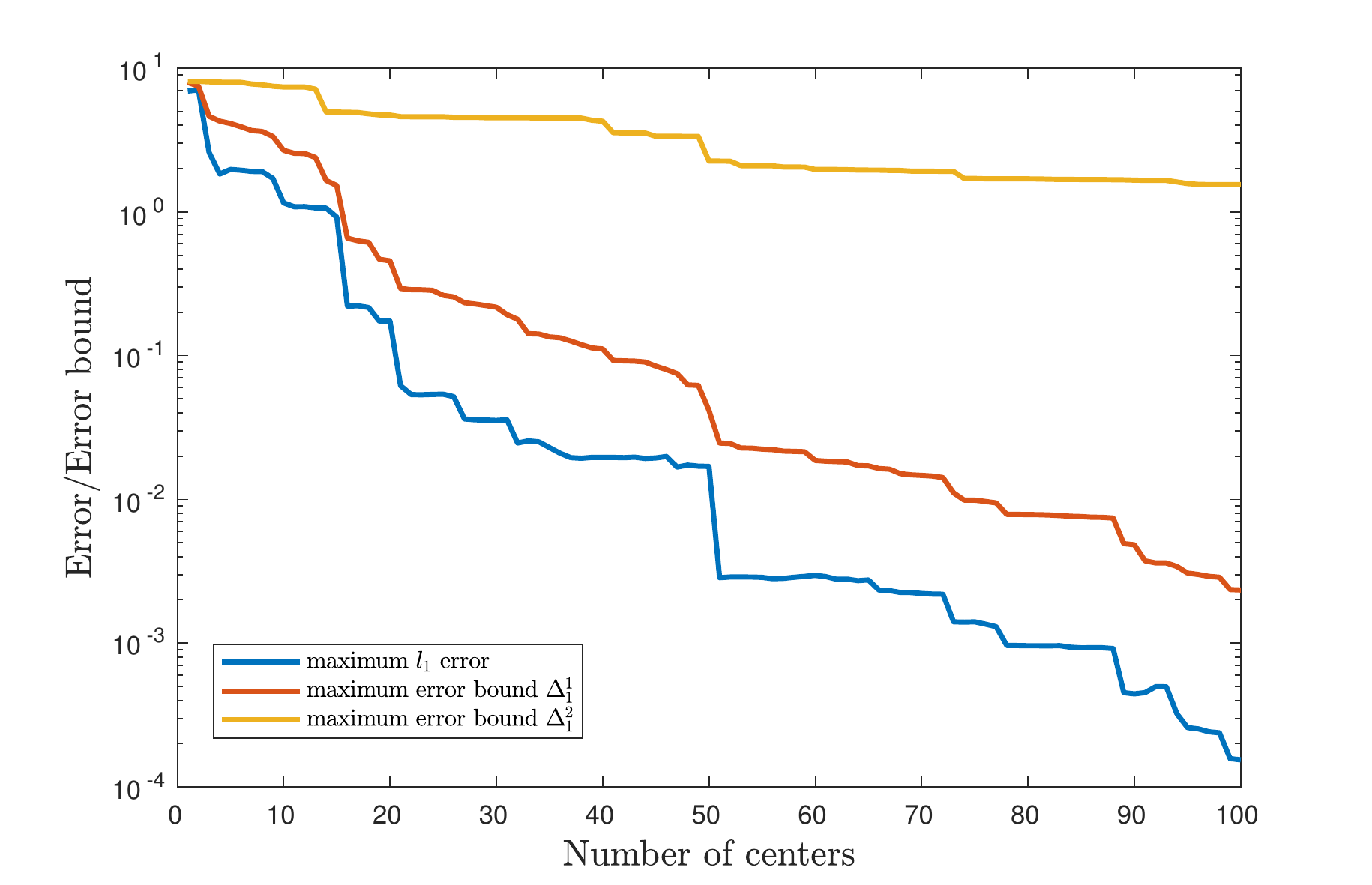}
\caption{Maximum pointwise error measured in the one norm for increasing number of centers.}
\label{figure:4}
\end{figure}


\section{Conclusion}
In this paper we recalled the concept of matrix-valued kernels and showed how they can be used to compute approximations or surrogate models
for which a-priori error estimate in various norms are available by means of the power-function. Furthermore, we introduced a new subclass of 
separable matrix-valued kernels, for which the power-function can be traced back to the power-functions of scalar-valued kernels. In an artificial
example for a low-dimensional output we illustrated how matrix-valued kernels can be used to encode correlations between function components
which leads to a significant improvement in the quality of the approximation.

Future work will investigate the selection of suitable centers via Greedy algorithms, where we obtained initial results in \cite{WH18a}.


\appendix
\section*{Appendix}
\begin{theorem*}
  Let $\cH_1, \dots, \cH_p$ be  RKHS with reproducing kernels $k_1,\dots,k_p$. Then $\cH = \bigoplus\limits_{i=1}^p \cH_i$ is a RKHS with reproducing kernel $k = \sum\limits_{i=1}^p k_i$ and norm given by
  \begin{equation*}
   \| f \|_{\cH}^2 = \min\left\{ \sum\limits_{i=1}^p \|f_i\|_{\cH_i}^2 \left| \right. \, f = \sum\limits_{i=1}^p f_i, \, f_i \in \cH_i \right\}
  \end{equation*}
\end{theorem*}

\begin{proof}
 By the principle of induction it is sufficient to consider the case $p = 2$. Therefore, let $ \cM := \cH_1 \times \cH_2$. One easily verifies that $\cM$ equipped with the inner product
 \begin{equation*}
  \langle (f_1,f_2), (g_1,g_2) \rangle_{\cM} = \langle f_1,g_1 \rangle_{\cH_1} + \langle f_2,g_2\rangle_{\cH_2}
 \end{equation*}
 is an RKHS with reproducing kernel $k_{\cM} = (k_1,k_2)$. Furthermore, let $S : \cM \rightarrow \cH_1 + \cH_2$ be given by
 \begin{equation*}
  S(f_1,f_2) = f_1 + f_2
 \end{equation*}
 and denote $\cN := S^{-1}(\{0\}) = \cH_1 \cap \cH_2$.
\end{proof}
Then $\cN$ is a closed subspace and thus $\cM = \cN \oplus \cN^{\perp}$. Therefore, $ T := S|_{\cN^{\perp}} : \cN^{\perp} \rightarrow \cH$ is a bijection and we equip $\cH$ with the inner product
\begin{equation*}
 \langle f,g \rangle_{\cH} = \langle T^{-1}(f),T^{-1}(g) \rangle_{\cM}.
\end{equation*}
For any arbitrary $f \in \cH$ we now have
\begin{equation*}
 S^{-1}(\{f\}) = T^{-1}(f) + \cN
\end{equation*}
and therefore
\begin{equation*}
 \| f \|^2_{\cH} = \|T^{-1}(f) \|_{\cM}^2 = \min \{ \|f_1\|_{\cH_1}^2 + \| f_2 \|_{\cH_2}^2 | f_1 + f_2 = f, \, f_1 \in \cH_1, f_2 \in \cH_2\}.T
\end{equation*}
It remains to show that $k = k_1 + k_2$ satisfies the reproducing property. By definition $k(\cdot,x)\alpha \in \cH$ is clear. Let $f \in \cH$ and let $(g_1,g_2) = T^{-1}(k(\cdot,x)\alpha)$. It now holds
that $(g_1 - k_1(\cdot,x)\alpha,g_2 - k_2(\cdot,x)\alpha) \in \cN$ and $(f_1,f_2):= T^{-1}(f) \in \cN^{\perp}$. Therefore,
\begin{align*}
 \langle f, k(\cdot,x)\alpha \rangle_{\cH} & = \langle T^{-1}(f), (g_1,g_2) \rangle_{\cM} \\
  & = \langle T^{-1}(f), (g_1 - k_1(\cdot,x)\alpha, g_2 - k_2(\cdot,x)\alpha) \rangle_{\cM} + \langle T^{-1}(f), (k_1(\cdot,x)\alpha,k_2(\cdot,x)\alpha) \rangle_{\cM}\\
  & = \langle (f_1,f_2), (k_1(\cdot,x)\alpha,k_2(\cdot,x)\alpha \rangle_{\cM} \\
  & = \langle f_1, k_1(\cdot,x)\alpha\rangle_{\cH_1} + \langle f_2, k_2(\cdot,x)\alpha\rangle_{\cH_2} \\
  & = f_1(x)^T\alpha + f_2(x)^T\alpha \\
  & = f(x)^T \alpha.
\end{align*}
Therefore, $k$ is the reproducing kernel of $\cH$.

\bibliography{anm}

\end{document}